\documentclass[10pt,a4paper,twoside]{amsart}

\usepackage{graphicx}
\usepackage{a4wide}

\newtheorem{definition}{Definition}[section]
\newtheorem{theorem}{Theorem}[section]
\newtheorem{lemma}{Lemma}[section]

\newtheorem*{maintheorem*}{Main Theorem}
\allowdisplaybreaks
\numberwithin{equation}{section}

\newcommand{\Kruzkov}{Kru\v{z}kov~}

\newcommand{\Set}[1]{\left\{#1\right\}}

\newcommand{\norm}[1]{\left\| #1 \right\|}

\newcommand{\eps}{\varepsilon}

\newcommand{\epsk}{\eps_k}

\newcommand{\ue}{u_\eps}

\newcommand{\uek}{u_{\eps_k}}

\newcommand{\Fe}{F_\eps}
\newcommand{\Pe}{P_\eps}
\newcommand{\ges}{g_\eps}

\newcommand{\Pek}{P_{\eps_k}}

\newcommand{\pt}{\partial_t}
\newcommand{\pT}{\partial_{\tau}}

\newcommand{\py}{\partial_y }

\newcommand{\px}{\partial_x }
\newcommand{\pxx}{\partial_{xx}^2}
\newcommand{\pxxx}{\partial_{xxx}^3}

\newcommand{\ptx}{\partial_{tx}^2}

\renewcommand{\i}{\ifmmode\mathit{\mathchar"7010 }\else\char"10 \fi}
\renewcommand{\j}{\ifmmode\mathit{\mathchar"7011 }\else\char"11 \fi}
\newcommand{\R}{\mathbb{R}}
\newcommand{\N}{\mathbb{N}}

\newcommand{\Hneg}{H_{\mathrm{loc}}^{-1}}
\newcommand{\CL}{\mathcal{L}}
\newcommand{\CLea}{\mathcal{L}}

\newcommand{\dtheta}{\, d\theta}

\newcommand{\dtau}{\, d\tau}

\newcommand{\sgn}[1]{\mathrm{sign}\left(#1\right)}

\newcommand{\ve}{v_\varepsilon}

{%

\begin{enumerate}}%
{\end{enumerate}}

{%

\begin{enumerate}}%
{\end{enumerate}}

\begin{document}\large

\title[The Ostrovsky--Hunter Equation]{Wellposedness of bounded solutions\\ of the non-homogeneous initial boundary\\ value problem for the Ostrovsky-Hunter equation}

\author[G. M. Coclite and L. di Ruvo]{Giuseppe Maria Coclite and Lorenzo di Ruvo}
\address[Giuseppe Maria Coclite and Lorenzo di Ruvo]
{\newline Department of Mathematics,   University of Bari, via E. Orabona 4, 70125 Bari,   Italy}
\email[]{giuseppemaria.coclite@uniba.it, lorenzo.diruvo@uniba.it}
\urladdr{http://www.dm.uniba.it/Members/coclitegm/}




\thanks{The authors would like to thank Prof. Fabio Ancona for suggesting the problem and Prof. Kenneth Hvistendahl Karlsen for many useful discussions.}

\begin{abstract}
The Ostrovsky-Hunter equation provides a model for small-amplitude long
waves in a rotating fluid of  finite depth. It is a nonlinear evolution equation.
In this paper  the welposedness of bounded solutions for a non-homogeneous initial boundary value problem associated to this equation is studied.
\end{abstract}

\maketitle

\section{Introduction}
\label{sec:intro}

The non-linear evolution equation
\begin{equation}
\label{eq:OHbeta}
\px(\pt u+u\px u-\beta \pxxx u)=\gamma u,
\end{equation}
with $\beta,\,\gamma \in \R$, was derived by Ostrovsky \cite{O} to model small-amplitude long
waves in a rotating fluid of a finite depth. This equation generalizes the
Korteweg-deVries equation (that corresponds to $\gamma=0$) by the additional term induced by
the Coriolis force. Mathematical properties of the Ostrovsky equation \eqref{eq:OHbeta} were studied
recently in many details, including the local and global well-posedness in energy space
\cite{GL, LM, LV, T}, stability of solitary waves \cite{LL, L, LV:JDE}, convergence of solutions in the
limit, $\gamma\to0$, of the Korteweg-deVries equation \cite{LL:07, LV:JDE}, and convergence of solutions in the
limit,  $\beta\to0$, of no high-frequency dispersion \cite{Cd}.

We shall consider the limit of no high-frequency dispersion $\beta=0$, therefore \eqref{eq:OHbeta} reads
\begin{equation}
\label{eq:OH}
\px(\pt u+u\px u)=\gamma u,\qquad t>0, \quad x>0.
\end{equation}
It is deduced considering two asymptotic expansions of the shallow water equations, first with respect to the rotation frequency and then with respect to the amplitude of the waves (see \cite{dR, HT}). It is known under different names such as the reduced Ostrovsky equation \cite{P, S}, the
Ostrovsky-Hunter equation \cite{B}, the short-wave equation \cite{H}, and the Vakhnenko equation
\cite{MPV, PV}.

We augment \eqref{eq:OH} with the boundary condition
\begin{equation}
\label{eq:boundary}
u(t,0)=g(t), \qquad t>0,
\end{equation}
and the initial datum
\begin{equation}
\label{eq:init}
u(0,x)=u_0(x), \qquad x>0,
\end{equation}
on which we assume that
\begin{equation}
\label{eq:assinit}
u_0\in L^{\infty}(0,\infty)\cap L^{1}(0,\infty).
\end{equation}
On the function
\begin{equation}
\label{eq:def-di-P0}
P_{0}(x)=\int_{0}^{x}u_{0}(y)dy,
\end{equation}
we assume that
\begin{equation}
\label{eq:l-2-di-P0}
\norm{P_{0}}^2_{L^2(0,\infty)}=\int_{0}^{\infty}\left(\int_{0}^{x} u_{0}(y)dy)\right)^2 dx < \infty.
\end{equation}
On the boundary datum $g(t)$, we assume that
\begin{equation}
\label{eq:ass-g}
g(t)\in W^{1,\infty}(0,\infty), \quad g(0)=0.
\end{equation}
Moreover, we assume that
\begin{equation}
\label{eq:gamma}
\gamma >0.
\end{equation}

Integrating \eqref{eq:OH} on $(0,x)$ we gain the integro-differential formulation of the initial-boundary value problem \eqref{eq:OH}, \eqref{eq:boundary},
  \eqref{eq:init} (see \cite{LPS})
\begin{equation}
\label{eq:OHw-u}
\begin{cases}
\pt u+u\px u =\gamma \int^x_0 u(t,y) dy,&\qquad t>0, \ x>0,\\
u(t,0)=g(t),& \qquad t>0,\\
u(0,x)=u_0(x), &\qquad x>0,
\end{cases}
\end{equation}
that is equivalent to
\begin{equation}
\label{eq:OHw}
\begin{cases}
\pt u+u\px u=\gamma P,&\qquad t>0, \ x>0 ,\\
\px P=u,&\qquad t>0, \ x>0,\\
u(t,0)=g(t), & \qquad t>0,\\
P(t,0)=0,& \qquad t>0,\\
u(0,x)=u_0(x), &\qquad x>0.
\end{cases}
\end{equation}

Due to the regularizing effect of the $P$ equation in \eqref{eq:OHw} we have that
\begin{equation}
\label{eq:OHsmooth}
    u\in L^{\infty}((0,T)\times(0,\infty))\Longrightarrow P\in L^{\infty}((0,T);W^{1_,\infty}(0,\infty)), \quad T>0.
\end{equation}
Therefore, if a map $u\in  L^{\infty}((0,T)\times(0,\infty)),\,T>0,$  satisfies, for every convex
map $\eta\in  C^2(\R)$,
\begin{equation}
\label{eq:OHentropy}
   \pt \eta(u)+ \px q(u)-\gamma\eta'(u) P\le 0, \qquad     q(u)=\int^u f'(\xi) \eta'(\xi)\, d\xi,
\end{equation}
in the sense of distributions, then \cite[Theorem 1.1]{CKK}
provides the existence of strong trace $u^\tau_0$ on the
boundary $x=0$.

We give the following definition of solution (see \cite{BRN}):

\begin{definition}
\label{def:sol}
We say that  $u\in  L^{\infty}((0,T)\times(0,\infty))$, $T>0$, is an entropy solution of the initial-boundary
value problem \eqref{eq:OH}, \eqref{eq:boundary}, and  \eqref{eq:init} if
for every nonnegative test function $\phi\in C^2(\R^2)$ with compact support, and $c\in\R$
\begin{equation}
\label{eq:ent2}
\begin{split}
\int_{0}^{\infty}\!\!\!\!\int_{0}^{\infty}\Big(\vert u - c\vert\pt\phi&+\sgn{u-c}\left(\frac{u^2}{2}-\frac{c^2}{2}\right)\px\phi\Big)dtdx\\
&+\gamma\int_{0}^{\infty}\!\!\!\!\int_{0}^{\infty}\sgn{u-c}P\phi dtdx\\
&+\int_{0}^{\infty}\sgn{g(t)-c}\left(\frac{(u^\tau_0(t))^2}{2}-\frac{c^2}{2}\right)\phi(t,0)dt\\
&+\int_{0}^{\infty}\vert u_{0}(x)-c\vert\phi(0,x)dx\geq 0,
\end{split}
\end{equation}
where $u^\tau_0(t)$ is the trace of $u$ on the boundary $x=0$.
\end{definition}

The main result of this paper  is the following theorem.
\begin{theorem}
\label{th:main}
Assume \eqref{eq:boundary}, \eqref{eq:init}, \eqref{eq:assinit}, \eqref{eq:def-di-P0}, \eqref{eq:l-2-di-P0}, \eqref{eq:ass-g} and \eqref{eq:gamma}.
The initial-boundary value problem
\eqref{eq:OH}, \eqref{eq:boundary} and \eqref{eq:init} possesses
an unique entropy solution $u$ in the sense of Definition \ref{def:sol}.
Moreover, if $u$ and $v$ are two entropy solutions \eqref{eq:OH}, \eqref{eq:boundary}, \eqref{eq:init} in the sense of Definition \ref{def:sol} the following inequality holds
 \begin{equation}
 \label{eq:stability}
\norm{u(t,\cdot)-v(t,\cdot)}_{L^1(0,R)}\le  e^{C(T) t}\norm{u(0,\cdot)-v(0,\cdot)}_{L^1(0,R+C(T)t)},
\end{equation}
for almost every $0<t<T$, $R>0$, and some suitable constant $C(T)>0$.
\end{theorem}

A similar result has been proved in \cite{CdK, dR} in the context of locally bounded solutions under the assumption $g\equiv 0$.

The paper is organized as follows. In Section \ref{sec:vv} we prove several a priori estimates on a vanishing viscosity approximation of \eqref{eq:OHw}.
Those play a key role in the proof of our main result, that is given in Section \ref{sec:proof}.

\section{Vanishing viscosity approximation}
\label{sec:vv}

Our existence argument is based on passing to the limit
in a vanishing viscosity approximation of \eqref{eq:OHw}.

Fix a small number $0<\eps<1$, and let $\ue=\ue (t,x)$ be the unique classical solution of the following mixed problem \cite{CHK:ParEll}
\begin{equation}
\label{eq:OHepsw}
\begin{cases}
\pt \ue+\ue\px\ue=\gamma\Pe+ \eps\pxx\ue,&\quad t>0,\ x>0,\\
-\eps\pxx\Pe+\px\Pe=\ue,&\quad t>0,\ x>0,\\
\ue(t,0)=\ges(t), &\quad t>0,\\
\Pe(t,0)=0,&\quad t>0,\\
\ue(0,x)=u_{\eps,0}(x),&\quad x>0,
\end{cases}
\end{equation}
where $u_{\eps,0}$ is a $C^\infty$ approximation of $u_{0}$ such that
\begin{equation}
\label{eq:u0eps}
\begin{split}
&\norm{u_{\eps,0}}_{L^2(0,\infty)}\le \norm{u_0}_{L^2(0,\infty)}, \quad \norm{u_{\eps,0}}_{L^{\infty}(0,\infty)}\le \norm{u_0}_{L^{\infty}(0,\infty)},\\
&\norm{P_{\eps,0}}^2_{L^2(0,\infty)}\le \norm{P_{0}}^2_{L^2(0,\infty)},\quad\eps^2\norm{\px P_{\eps,0}}^2_{L^2(0,\infty)}\le C_{0},\\
&\norm{\ges}_{L^{\infty}(0,\infty)}+\norm{\ges'}_{L^{\infty}(0,\infty)}\le C_0, \quad \ges(0)=0,
\end{split}
\end{equation}
and $C_0$ is a constant independent on $\eps$.

Let us prove some a priori estimates on $\ue$ and $\Pe$, denoting with $C_0$ the constants which depend on the initial data, and $C(T)$ the constants which depend also on $T$.
\begin{lemma}
\label{lm:cns}
For each $t\in (0,\infty)$,
\begin{equation}
\label{eq:P-pxP-intfy}
\Pe(t,\infty)=\px\Pe(t,\infty)=0.
\end{equation}
Moreover,
\begin{equation}
\begin{split}
\label{eq:equ-L2-stima}
\eps^2\norm{\pxx\Pe(t,\cdot)}^2_{L^2(0,\infty)}&+\eps(\px\Pe(t,0))^2\\
&+ \norm{\px\Pe(t,\cdot)}^2_{L^2(0,\infty)}=\norm{\ue(t,\cdot)}^2_{L^2(0,\infty)}.
\end{split}
\end{equation}
\end{lemma}

\begin{proof}
We begin by proving that \eqref{eq:P-pxP-intfy} holds true.

Differentiating the first equation of \eqref{eq:OHepsw} with respect to $x$, we have
\begin{equation}
\label{eq:pxu}
\px(\pt \ue+\ue\px\ue-\eps\pxx\ue)=\gamma\px\Pe.
\end{equation}
For the the smoothness of $\ue$, it follows from \eqref{eq:OHepsw} and \eqref{eq:pxu} that
\begin{align*}
&\lim_{x\to\infty}\pt \ue+\ue\px \ue-\eps\pxx\ue=\gamma\Pe(t,\infty)=0,\\
&\lim_{x\to\infty}\px(\pt \ue+\ue\px \ue-\eps\pxx\ue)=\gamma\px\Pe(t,\infty)=0,
\end{align*}
which gives \eqref{eq:P-pxP-intfy}.

Let us show that \eqref{eq:equ-L2-stima} holds true.
Squaring the equation for $\Pe$ in \eqref{eq:OHepsw}, we get
\begin{equation*}
\eps^2(\pxx\Pe)^2+(\px\Pe)^2 - \eps\px((\px\Pe)^2)=\ue^2.
\end{equation*}
Therefore, \eqref{eq:equ-L2-stima} follows from \eqref{eq:OHepsw}, \eqref{eq:P-pxP-intfy} and an integration on $(0,\infty)$.
\end{proof}

\begin{lemma}
\label{lm:2}
For each $t\in(0,\infty)$,
\begin{align}
\label{eq:int-u}
\int_{0}^{\infty}\ue(t,x) dx &= \eps\px\Pe(t,0),\\
\label{eq:L-infty-Px}
\sqrt{\eps}\norm{\px\Pe(t, \cdot)}_{L^{\infty}(0,\infty)}&\le \norm{u(t,\cdot)}_{L^2(0,\infty)},\\
\label{eq:uP}
\int_{0}^{\infty}\ue(t,x)\Pe(t,x) dx&\le \norm{u(t,\cdot)}^2_{L^2(0,\infty)}.
\end{align}
\end{lemma}

\begin{proof}
Integrating on $(0,\infty)$ the equation for $\Pe$ in \eqref{eq:OHepsw}, for \eqref{eq:P-pxP-intfy}, we have
\begin{equation*}
\int_{0}^{\infty}\ue(t,x)dx=\eps\px\Pe(t,0),
\end{equation*}
that is \eqref{eq:int-u}.

Let us show that \eqref{eq:L-infty-Px} holds true. Observe that
\begin{equation*}
0\le (-\eps \pxx\Pe + \px\Pe)^2= \eps^2(\pxx\Pe)^2 +(\px\Pe)^2 - \eps\px((\px\Pe)^2),
\end{equation*}
that is,
\begin{equation}
\label{eq:equa-pxP}
\eps\px((\px\Pe)^2)\le \eps^2(\pxx\Pe)^2 +(\px\Pe)^2.
\end{equation}
Integrating \eqref{eq:equa-pxP} in $(0,x)$, we have
\begin{equation}
\label{eq:equa-pxP1}
\begin{split}
\eps(\px\Pe)^2 - \eps(\px\Pe(t,0))^2 &\le \eps^2\int_{0}^{x}(\pxx\Pe)^2 dx +\int_{0}^{x}(\px\Pe)^2 dx\\
&\le \eps^2\int_{0}^{\infty}(\pxx\Pe)^2 dx +\int_{0}^{\infty}(\px\Pe)^2 dx.
\end{split}
\end{equation}
It follows from \eqref{eq:equ-L2-stima} and \eqref{eq:equa-pxP1} that
\begin{equation*}
\eps(\px\Pe)^2\le \eps^2\int_{0}^{\infty}(\pxx\Pe)^2 dx +\int_{0}^{\infty}(\px\Pe)^2 dx+ \eps(\px\Pe(t,0))^2\le \norm{\ue(t,\cdot)}^2_{L^2(0,\infty)}.
\end{equation*}
Therefore,
\begin{equation*}
\sqrt{\eps}\vert \px\Pe(t,x)\vert \le \norm{\ue(t,\cdot)}_{L^2(0,\infty)},
\end{equation*}
which gives \eqref{eq:L-infty-Px}.

Finally, we prove \eqref{eq:uP}.
Multiplying by $\Pe$ the equation for $\Pe$ of \eqref{eq:OHepsw}, we get
\begin{equation*}
-\eps\Pe\pxx\Pe + \Pe\px\Pe= \ue\Pe.
\end{equation*}
An integration on $(0,\infty)$ and \eqref{eq:P-pxP-intfy} give
\begin{align*}
\int_{0}^{\infty}\ue\Pe dx=&\frac{1}{2}\int_{0}^{\infty}\px(\Pe)^2 dx - \eps \int_{0}^{\infty}\Pe\pxx\Pe dx\\
=&-\eps\int_{0}^{\infty}\Pe\pxx\Pe dx=\eps\int_{0}^{\infty}(\px\Pe)^2 dx,
\end{align*}
that is
\begin{equation*}
\int_{0}^{\infty}\ue\Pe dx = \eps\int_{0}^{\infty}(\px\Pe)^2 dx.
\end{equation*}
Since $0<\eps <1$, for \eqref{eq:equ-L2-stima}, we have \eqref{eq:uP}.
\end{proof}
Let us consider the following function
\begin{equation}
\label{eq:def-di-v}
\ve(t,x)= \ue(t,x)-\ges(t)\chi(x),
\end{equation}
where $\chi\in C^{\infty}(0,\infty)$ is a cut-off function such that
\begin{equation}
\label{eq:def-di-chi}
\begin{split}
\chi(0)&=1,\\
\norm{\chi}_{L^{\infty}(0,\infty)}, \norm{\chi'}_{L^{\infty}(0,\infty)}&\le C_0,\\
\norm{\chi}^2_{L^2(0,\infty)}, \norm{\chi'}^2_{L^2(0,\infty)}& \le C_0.
\end{split}
\end{equation}
Therefore, it follows from \eqref{eq:OHepsw}, \eqref{eq:def-di-v} and \eqref{eq:def-di-chi} that
\begin{equation}
\label{eq:v-in-0}
\ve(t,0)=\ges(t)-\ges(t)=0.
\end{equation}
For \eqref{eq:u0eps},
\begin{equation*}
\ve(0,x)=v_{\eps,0}(x)=\ue(0,x)=u_{\eps,0}(x).
\end{equation*}
Therefore, again by  \eqref{eq:u0eps},
\begin{equation}
\label{eq:con-l-2-v}
\norm{v_{\eps,0}}_{L^2(0,\infty)}= \norm{u_{\eps,0}}_{L^2(0,\infty)}.
\end{equation}
Moreover,
\begin{equation}
\label{eq:from-u-to-v}
\begin{split}
\pt\ue& = \pt\ve + \ges'(t)\chi,\\
\px\ue&= \px\ve +\ges(t)\chi',\\
\pxx\ue&=\pxx\ve+\ges(t)\chi''.
\end{split}
\end{equation}
Thus, for \eqref{eq:OHepsw}, \eqref{eq:def-di-v} and \eqref{eq:from-u-to-v}, we have
\begin{equation*}
\pt\ve + \ges'(t)\chi + (\ve+\ges(t)\chi)(\px\ve +\ges(t)\chi')=\gamma\Pe +\eps(\pxx\ve+\ges(t)\chi''),
\end{equation*}
that is,
\begin{equation}
\label{eq:equ-di-v}
\begin{split}
\pt\ve &+ \ve\px\ve+ \ges(t)\ve\chi'+ \ges(t)\chi\px\ve\\
=&\gamma\Pe +\eps\pxx\ve+\eps \ges(t)\chi'' -\ges'(t)\chi - \ges^2(t)\chi\chi'.
\end{split}
\end{equation}

\begin{lemma}
\label{lm:P-in-v}
For each $t>0$, we have that
\begin{align}
\label{eq:nor-u-in-v}
\norm{\ue(t,\cdot)}^2_{L^2(0,\infty)}&\le 2 \norm{\ve(t,\cdot)}^2_{L^2(0,\infty)} + C_{0},\\
\label{eq:nor-Px-in-v}
\norm{\px\Pe(t,\cdot)}^2_{L^2(0,\infty)}&\le 2 \norm{\ve(t,\cdot)}^2_{L^2(0,\infty)} + C_{0},\\
\label{eq:Pv}
\int_{0}^{\infty}\Pe(t,x)\ve(t,x) dx &\le C_{0} \norm{\ve(t,\cdot)}^2_{L^2(0,\infty)} + C_{0}.
\end{align}
\end{lemma}

\begin{proof}
We begin by observing that, for \eqref{eq:def-di-v}, we get
\begin{equation}
\label{eq:u-in-v}
\ue=\ve + \ges(t)\chi.
\end{equation}
Squaring \eqref{eq:u-in-v}, we have
\begin{equation*}
\ue^2= \ve^2 +2 \ges(t)\ve\chi+\ges^2(t)\chi^2.
\end{equation*}
Due to the Young's inequality,
\begin{equation*}
2\vert\ges(t)\ve\chi\vert \le \ve^2 +  \ges^2(t)\chi^2.
\end{equation*}
Therefore,
\begin{equation*}
\ue^2\le 2\ve^2+2\ges^2(t)\chi^2.
\end{equation*}
\eqref{eq:u0eps}, \eqref{eq:def-di-chi} and an integration on $(0,\infty)$ give \eqref{eq:nor-u-in-v}.
\eqref{eq:nor-Px-in-v} follows from \eqref{eq:equ-L2-stima} and \eqref{eq:nor-u-in-v}.

Let us show \eqref{eq:Pv}.
We observe that, for \eqref{eq:P-pxP-intfy} and \eqref{eq:def-di-v},
\begin{align*}
\int_{0}^{\infty}\Pe\ve dx &= \int_{0}^{\infty}\Pe\ue dx - \ges(t)\int_{0}^{\infty}\Pe\chi dx\\
&= \int_{0}^{\infty}\Pe\ue dx + \ges(t)\int_{0}^{\infty}\px\Pe\chi' dx.
\end{align*}
Thanks to \eqref{eq:u0eps}, \eqref{eq:def-di-chi} and Young's inequality,
\begin{equation}
\begin{split}
\label{eq:pxP-chi}
&\left\vert\ges(t)\int_{0}^{\infty}\px\Pe\chi' dx\right\vert \le \vert\ges(t)\vert\int_{0}^{\infty}\vert\px\Pe\vert\vert\chi'\vert dx\\
&\qquad \le \frac{C_{0}}{2}\norm{\px\Pe(t,\cdot)}^2_{L^2(0,\infty)} +\frac{C_{0}}{2}\norm{\chi'}^2_{L^2(0,\infty)}\\
&\qquad\le C_{0}\norm{\px\Pe(t,\cdot)}^2_{L^2(0,\infty)}+C_{0}.
\end{split}
\end{equation}
Hence, for \eqref{eq:uP}, \eqref{eq:nor-u-in-v}, \eqref{eq:nor-Px-in-v} and \eqref{eq:pxP-chi},
\begin{align*}
\int_{0}^{\infty}\Pe\ve dx &\leq 2 \norm{\ve(t,\cdot)}^2_{L^2(0,\infty)} + C_{0} + C_{0} \norm{\ve(t,\cdot)}^2_{L^2(0,\infty)} + C_{0}\\
&\le C_{0}\norm{\ve(t,\cdot)}^2_{L^2(0,\infty)} + C_{0},
\end{align*}
that is \eqref{eq:Pv}.
\end{proof}

\begin{lemma}\label{lm:stima-l-2}
For each $t>0$, the inequality holds
\begin{equation}
\label{eq:stima-l-2-v}
\norm{\ve(t,\cdot)}^2_{L^2(0,\infty)}+ \eps e^{C_{0} t}\int_{0}^{t}e^{-C_{0} s}\norm{\px\ve(s,\cdot)}^2_{L^2(0,\infty)}ds\le C_{0}e^{C_{0}t}(1+t).
\end{equation}
In particular, we have
\begin{align}
\label{eq:l-2-u}
\norm{\ue(t,\cdot)}^2_{L^2(0,\infty)}&\le  C_{0}\left(e^{C_{0}t}(1+t) + 1\right),\\
\label{eq:ux-l-2}
\eps\int_{0}^{t}\norm{\px\ue(s,\cdot)}^2_{L^2(0,\infty)}ds&\le C_{0}\left(e^{C_{0}t}(1+t)+t\right).
\end{align}
Moreover,
\begin{equation}
\label{eq:10021}
\begin{split}
\eps\norm{\pxx\Pe(t,\cdot)}_{L^2(0,\infty)}, \norm{\px\Pe(t,\cdot)}_{L^2(0,\infty)}&\le \sqrt{C_{0}\left(e^{C_{0}t}(1+t) +1\right)},\\
\sqrt{\eps}\vert\px\Pe(t,0)\vert, \sqrt{\eps}\norm{\px\Pe(t, \cdot)}_{L^{\infty}(0,\infty)}&\le \sqrt{C_{0}\left(e^{C_{0}t}(1+t) + 1\right)}.
\end{split}
\end{equation}
\end{lemma}
\begin{proof}
Let $t>0$. Multiplying \eqref{eq:equ-di-v} by $\ve$, we have
\begin{equation}
\label{eq:ptv-v}
\begin{split}
\ve\pt\ve &+ \ve^2\px\ve+ \ges(t)\ve^2\chi'+ \ges(t)\ve\chi\px\ve\\
=&\gamma\Pe\ve +\eps\ve\pxx\ve+\eps \ges(t)\ve\chi'' -\ges'(t)\ve\chi - \ges^2(t)\ve\chi\chi'.
\end{split}
\end{equation}
Since,
\begin{align*}
\int_{0}^{\infty}\ve\pt\ve dx &= \frac{1}{2}\frac{d}{dt}\norm{\ve(t,\cdot)}^2_{L^2(0,\infty)},\\
\ges(t)\int_{0}^{\infty}\ve\chi\px\ve dx&=-\frac{\ges(t)}{2}\int_{0}^{\infty}\ve^2\chi' dx ,\\
\eps\int_{0}^{\infty}\ve\pxx\ve dx &= -\eps \norm{\px\ve(t,\cdot)}^2_{L^2(0,\infty)},\\
\eps \ges(t)\int_{0}^{\infty}\ve\chi'' dx&= -\eps\ges(t)\int_{0}^{\infty}\px\ve\chi' dx,
\end{align*}
integrating \eqref{eq:ptv-v} on $(0,\infty)$,
\begin{equation}
\begin{split}
\label{eq:equ-l-2-v-1}
&\frac{1}{2}\frac{d}{dt}\norm{\ve(t,\cdot)}^2_{L^2(0,\infty)}+ \eps \norm{\px\ve(t,\cdot)}^2_{L^2(0,\infty)}\\
&\quad = -\ges(t)\int_{0}^{\infty}\ve^2 \chi dx +\frac{\ges(t)}{2}\int_{0}^{\infty}\ve^2\chi' dx\\
&\qquad +\gamma\int_{0}^{\infty}\Pe\ve dx -\eps\ges(t)\int_{0}^{\infty}\px\ve\chi' dx\\
&\qquad  -\ges'(t)\int_{0}^{\infty}\ve\chi dx - \ges^2(t)\int_{0}^{\infty}\ve\chi\chi' dx.
\end{split}
\end{equation}
Due to \eqref{eq:u0eps}, \eqref{eq:def-di-chi} and Young's inequality,
\begin{align*}
&\eps\left\vert\ges(t)\int_{0}^{\infty}\px\ve\chi' dx\right\vert\le \eps \vert \ges(t)\vert \int_{0}^{\infty}\left\vert\frac{\px\ve}{D_1}\right\vert\left\vert\chi' D_1\right\vert dx\\
&\quad \le \eps \frac{C_{0}}{2D_1^2}\norm{\px\ve(t,\cdot)}^2_{L^2(0,\infty)} + \frac{D_1^2}{2}\norm{\chi'}^2_{L^2(0,\infty)}\\
&\quad \le \eps \frac{C_{0}}{2D_1^2}\norm{\px\ve(t,\cdot)}^2_{L^2(0,\infty)} +D_1^2 C_{0},\\
&\left\vert\ges'(t)\int_{0}^{\infty}\ve\chi dx\right\vert \le \vert \ges'(t)\vert \int_{0}^{\infty}\vert\ve\vert\vert\chi\vert dx\\
&\quad \le \frac{C_{0}}{2}\norm{\ve(t,\cdot)}^2_{L^2(0,\infty)}+\frac{C_{0}}{2}\norm{\chi}^2_{L^2(0,\infty)}\\
&\quad \le C_{0}\norm{\ve(t,\cdot)}^2_{L^2(0,\infty)}+C_{0},\\
&\ges^2(t)\left\vert\int_{0}^{\infty}\ve\chi\chi' dx\right\vert \le \ges^2(t) \int_{0}^{\infty}\vert\ve\vert\vert\chi\vert\vert\chi'\vert dx\\
& \quad \le \frac{C_{0}\norm{\chi'}_{L^{\infty}(0,\infty)}}{2}\left(\norm{\ve(t,\cdot)}^2_{L^2(0,\infty)}+\norm{\chi}^2_{L^2(0,\infty)}\right)\\
&\quad \le C_{0}\norm{\ve(t,\cdot)}^2_{L^2(0,\infty)}+C_{0},
\end{align*}
where $D_1$ is a positive constant that will be specified later.

Moreover, again by \eqref{eq:u0eps} and  \eqref{eq:def-di-chi},
\begin{align*}
&\left\vert \ges(t)\int_{0}^{\infty}\ve^2 \chi dx\right\vert \le \vert \ges(t)\vert \int_{0}^{\infty}\ve^2 \vert\chi\vert dx\\
&\quad\qquad \le C_{0} \norm{\chi}_{L^{\infty}(0,\infty)}\norm{\ve(t,\cdot)}^2_{L^2(0,\infty)}\le C_{0}\norm{\ve(t,\cdot)}^2_{L^2(0,\infty)},\\
&\left\vert\frac{\ges(t)}{2}\int_{0}^{\infty}\ve^2\chi' dx\right\vert\le \frac{\vert \ges(t)\vert}{2} \int_{0}^{\infty}\ve^2\vert\chi'\vert dx\\
&\quad\qquad\le C_{0}\norm{\chi'}_{L^{\infty}(0,\infty)}\norm{\ve(t,\cdot)}^2_{L^2(0,\infty)}\le C_{0}\norm{\ve(t,\cdot)}^2_{L^2(0,\infty)}.
\end{align*}
It follows from \eqref{eq:Pv} and \eqref{eq:equ-l-2-v-1} that
\begin{align*}
&\frac{d}{dt}\norm{\ve(t,\cdot)}^2_{L^2(0,\infty)}+ \eps \left (2 - \frac{C_{0}}{D_1^2}\right)\norm{\px\ve(t,\cdot)}^2_{L^2(0,\infty)}\\
&\quad \le \gamma C_{0} \norm{\ve(t,\cdot)}^2_{L^2(0,\infty)}+ 8C_{0}\norm{\ve(t,\cdot)}^2_{L^2(0,\infty)}\\
&\qquad +2\gamma C_{0} +C_{0} +D_1^2 C_{0},
\end{align*}
that is
\begin{align*}
&\frac{d}{dt}\norm{\ve(t,\cdot)}^2_{L^2(0,\infty)}+ \eps \left (2 - \frac{C_{0}}{D_1^2}\right)\norm{\px\ve(t,\cdot)}^2_{L^2(0,\infty)}\\
&\quad \le C_{0}\norm{\ve(t,\cdot)}^2_{L^2(0,\infty)}+C_{0}+D_1^2 C_{0}.
\end{align*}
Choosing $D_1^2=C_{0}$, we get
\begin{equation*}
\frac{d}{dt}\norm{\ve(t,\cdot)}^2_{L^2(0,\infty)}+ \eps\norm{\px\ve(t,\cdot)}^2_{L^2(0,\infty)}\le C_{0}\norm{\ve(t,\cdot)}^2_{L^2(0,\infty)}+C_{0}.
\end{equation*}
Gronwall's Lemma and \eqref{eq:con-l-2-v} give
\begin{align*}
&\norm{\ve(t,\cdot)}^2_{L^2(0,\infty)}+\eps e^{C_{0}t}\int_{0}^{t}e^{-C_{0}s}\norm{\px\ve(s,\cdot)}^2_{L^2(0,\infty)}ds\\
&\quad \le \norm{u_0}_{L^2(0,\infty)}e^{C_{0}t}+C_{0}e^{C_{0}t}\int_{0}^{t}e^{-C_{0}s}ds\le \norm{u_0}_{L^2(0,\infty)}e^{C_{0}t}+C_{0}te^{C_{0}t},
\end{align*}
which gives \eqref{eq:stima-l-2-v}.

Let us show that \eqref{eq:ux-l-2} holds true.
We begin by observing that, \eqref{eq:from-u-to-v} and an multiplication by $\sqrt{\eps}$ give
\begin{equation}
\label{eq:ux-12-1}
\sqrt{\eps}\px\ue=\sqrt{\eps}\px\ve + \sqrt{\eps}\ges(t)\chi'.
\end{equation}
Squaring \eqref{eq:ux-12-1}, we have
\begin{equation*}
\eps(\px\ue)^2 = \eps(\px\ve)^2+2\eps\ges(t)\px\ve\chi' +\eps\ges^2(t)(\chi')^2.
\end{equation*}
Due to Young's inequality,
\begin{equation*}
2\eps\vert\ges(t)\px\ve\chi'\vert \le \eps(\px\ve)^2 +  \eps\ges^2(t)(\chi')^2.
\end{equation*}
Therefore, since $0<\eps<1$,
\begin{equation*}
\eps(\px\ue)^2 \le 2\eps(\px\ve)^2+ 2\ges^2(t)(\chi')^2.
\end{equation*}
An integration on $(0,\infty)$, \eqref{eq:u0eps} and \eqref{eq:def-di-chi} give
\begin{equation}
\label{eq:ux-14}
\eps\norm{\px\ue(t,\cdot)}^2_{L^2(0,\infty)}\le 2\eps\norm{\px\ve(t,\cdot)}^2_{L^2(0,\infty)} +C_{0}.
\end{equation}
Integrating \eqref{eq:ux-14} on $(0,t)$, we get
\begin{equation}
\label{eq:ux-15}
\begin{split}
&\eps\int_{0}^{t}\norm{\px\ue(s,\cdot)}^2_{L^2(0,\infty)}ds\le 2\eps\int_{0}^{t}\norm{\px\ve(s,\cdot)}^2_{L^2(0,\infty)}ds +C_{0}t\\
&\quad \le 2\eps e^{C_{0}t}\int_{0}^{t}e^{-C_{0}s}\norm{\px\ve(s,\cdot)}^2_{L^2(0,\infty)}ds +C_{0}t.
\end{split}
\end{equation}
\eqref{eq:ux-l-2} follows from \eqref{eq:stima-l-2-v} and \eqref{eq:ux-15}.

Finally, \eqref{eq:10021} follows from \eqref{eq:equ-L2-stima}, \eqref{eq:L-infty-Px} and  \eqref{eq:l-2-u}.
\end{proof}

\begin{lemma}\label{lm:def-di-f}
Let us consider the following function
\begin{equation}
\label{eq:def-di-f}
\Fe(t,x)=\int_{0}^{x} \Pe(t,y)dy \quad t>0,\ x>0.
\end{equation}
We have that
\begin{equation}
\label{eq:lim-di-f}
\lim_{x\to\infty}\Fe(t,x)=\int_{0}^{\infty}\Pe(t,x)dx=\frac{\eps}{\gamma}\ptx\Pe(t,0)+\frac{\eps}{\gamma}\px\ue(t,0)-\frac{1}{2\gamma}\ges^2(t).
\end{equation}
\end{lemma}
\begin{proof}
Integrating on $(0,x)$ the first equation of \eqref{eq:OHepsw}, we get
\begin{equation}
\label{eq:int-in-x}
\int_{0}^{x}\pt\ue(t,y) dy + \frac{1}{2}\ue^2(t,x)-\frac{1}{2}\ges^2(t)  - \eps\px\ue(t,x) + \eps\px\ue(t,0)= \gamma\int_{0}^{x}\Pe(t,y) dy.
\end{equation}
It follows from the regularity of $\ue$ that
\begin{equation}
\label{eq:1212}
\lim_{x\to\infty}\left(\frac{1}{2}\ue^2(t,x)-\eps\px\ue(t,x)\right)=0.
\end{equation}
For \eqref{eq:int-u}, we have that
\begin{equation}
\label{eq:12123}
\lim_{x\to\infty}\int_{0}^{x}\pt\ue(t,y)dy=\int_{0}^{\infty}\pt\ue(t,x) dx = \frac{d}{dt}\int_{0}^{\infty}\ue(t,x)dx=\eps\ptx\Pe(t,0).
\end{equation}
\eqref{eq:int-in-x}, \eqref{eq:1212} and \eqref{eq:12123} give \eqref{eq:lim-di-f}.
\end{proof}

\begin{lemma}\label{lm:P-infty}
Let $0<t<T$. There exists a function $C(T)>0$, independent on $\eps$, such that
\begin{align}
\label{eq:l-infty-P-1}
\norm{\Pe}_{L^{\infty}(I_T)} &\le C(T),\\
\label{eq:l-2-P-1}
\norm{\Pe(t,\cdot)}_{L^2(0,\infty)}&\le C(T),\\
\label{eq:L-2-px-P-1}
\eps\norm{\px\Pe(t,\cdot)}_{L^2(0,\infty)}&\le C(T),\\
\label{eq:l-2-px-u}
e^{2\gamma t}\int_{0}^{t} e^{-2\gamma s}\left(\eps\ptx\Pe(s,0)+\eps\px\ue(s,0)-\frac{1}{2}\ges^2(s)\right)^2 ds &\le C(T),
\end{align}
where
\begin{equation}
\label{eq:def-di-I}
I_T=(0,T)\times(0,\infty).
\end{equation}
In particular, we have
\begin{equation}
\label{eq:pt-px-P}
\eps\left\vert\int_{0}^{t}\!\!\!\int_{0}^{\infty}\Pe\ptx\Pe ds dx\right\vert\le C(T), \quad 0<t<T.
\end{equation}
\end{lemma}
\begin{proof}
Let $0<t<T$. We begin by observing that, integrating in $(0,x)$ the second equation of \eqref{eq:OHepsw}, we get
\begin{equation}
\label{eq:P-int-in-0}
\Pe(t,x)=\int_{0}^{x}\ue(t,y)dy + \eps \px\Pe(t,x)- \eps\px\Pe(t,0).
\end{equation}
Differentiating with respect to $t$, we have that
\begin{align*}
\pt\Pe(t,x)&=\frac{d}{dt}\int_{0}^{x}\ue(t,y)dy + \eps \ptx\Pe(t,x)- \eps\ptx\Pe(t,0)\\
&=\int_{0}^{x}\pt\ue(t,x) +\eps \ptx\Pe(t,x)- \eps\ptx\Pe(t,0).
\end{align*}
It follows from \eqref{eq:def-di-f} and \eqref{eq:int-in-x} that
\begin{equation}
\label{eq:equat-per-P}
\begin{split}
\pt\Pe(t,x)=&\gamma\Fe(t,x) -\frac{1}{2}\ue^2(t,x))+\frac{1}{2}\ges^2(t)+\eps\px\ue(t,x)\\
&-\eps\px\ue(t,0)+ \eps \ptx\Pe(t,x)- \eps\ptx\Pe(t,0).
\end{split}
\end{equation}
Multiplying \eqref{eq:equat-per-P} by $\Pe - \eps\px\Pe$, we have that
\begin{equation}
\label{eq:1234}
\begin{split}
(\Pe - \eps\px\Pe)\pt\Pe= &\gamma(\Pe - \eps\px\Pe) \Fe-\frac{1}{2}(\Pe - \eps\px\Pe)\ue^2\\
&+\frac{1}{2}(\Pe - \eps\px\Pe)\ges^2(t) -\eps(\Pe - \eps\px\Pe)\px\ue(t,0)\\
& +\eps(\Pe -\eps\px\Pe)\px\ue +\eps(\Pe -\eps\px\Pe) \ptx\Pe\\
& - \eps(\Pe -\eps\px\Pe)\ptx\Pe(t,0).
\end{split}
\end{equation}
Integrating \eqref{eq:1234} on $(0,x)$, for \eqref{eq:OHepsw}, we get
\begin{equation}
\label{eq:1235}
\begin{split}
&\int_{0}^{x}\Pe\pt\Pe dy -\eps \int_{0}^{x}\px\Pe\pt\Pe dy\\
&\quad=\gamma\int_{0}^{x}\Pe\Fe dy - \eps\int_{0}^{x}\Fe\px\Pe dy - \frac{1}{2}\int_{0}^{x}\Pe \ue^2 dy\\
&\qquad +\frac{\eps}{2}\int_{0}^{x}\px\Pe \ue^2dy +\frac{1}{2}\ges^2(t)\int_{0}^{x}\Pe dy -\frac{\eps}{2}\ges^2(t)\Pe\\
&\qquad -\eps\px\ue(t,0)\int_{0}^{y} \Pe dx+\eps^2\px\ue(t,0)\Pe+\eps\int_{0}^{x} \Pe\px\ue dy \\
&\qquad -\eps^2\int_{0}^{x}\px\Pe\px\ue dy + \eps\int_{0}^{x}\Pe\ptx\Pe dy-\eps^2\int_{0}^{x}\px\Pe\ptx\Pe dy\\
&\qquad  -\eps\ptx\Pe(t,0)\int_{0}^{x}\Pe dy +\eps^2\ptx\Pe(t,0)\Pe.
\end{split}
\end{equation}
We observe that, for \eqref{eq:OHepsw},
\begin{equation}
\label{eq:int-by-part}
-\eps \int_{0}^{x}\px\Pe\pt\Pe dy=-\eps\Pe\pt\Pe + \eps\int_{0}^{x}\Pe\ptx\Pe dy.
\end{equation}
Therefore, \eqref{eq:1235} and \eqref{eq:int-by-part} give
\begin{equation}
\begin{split}
\label{eq:P-P-x}
&\int_{0}^{x}\Pe\pt\Pe dy+\eps^2\int_{0}^{x}\px\Pe\ptx\Pe dy\\
&\quad=\eps\Pe\pt\Pe + \gamma\int_{0}^{x}\Pe\Fe dy- \eps\int_{0}^{x}\Fe\px\Pe dy\\
&\qquad - \frac{1}{2}\int_{0}^{x}\Pe \ue^2 dy +\frac{\eps}{2}\int_{0}^{x}\px\Pe \ue^2dy +\frac{1}{2}\ges^2(t)\int_{0}^{x}\Pe dy \\
&\qquad -\frac{\eps}{2}\ges^2(t)\Pe-\eps\px\ue(t,0)\int_{0}^{y} \Pe dx +\eps^2\px\ue(t,0)\Pe \\
&\qquad +\eps\int_{0}^{x} \Pe\px\ue dy -\eps^2\int_{0}^{x}\px\Pe\px\ue dy-\eps\ptx\Pe(t,0)\int_{0}^{x}\Pe dy\\
&\qquad +\eps^2\ptx\Pe(t,0)\Pe.
\end{split}
\end{equation}
Since
\begin{align*}
\int_{0}^{\infty}\Pe\pt\Pe dx &=\frac{1}{2}\frac{d}{dt}\int_{0}^{\infty}\Pe^2dx,\\
\eps^2\int_{0}^{\infty}\ptx\Pe\px\Pe dx &= \frac{\eps^2}{2}\frac{d}{dt}\int_{0}^{\infty}(\px\Pe)^2dx,
\end{align*}
when $x\to\infty$, for \eqref{eq:P-pxP-intfy} and \eqref{eq:P-P-x}, we have that
\begin{equation}
\label{eq:12312}
\begin{split}
&\frac{1}{2}\frac{d}{dt}\int_{0}^{\infty}\Pe^2dx+\frac{\eps^2}{2}\frac{d}{dt}\int_{0}^{\infty}(\px\Pe)^2dx\\
&\quad= \gamma\int_{0}^{\infty}\Pe\Fe dx - \eps\gamma\int_{0}^{\infty}\px\Pe\Fe dx -\frac{1}{2}\int_{0}^{\infty} \Pe \ue^2dx\\
&\qquad +\frac{\eps}{2} \int_{0}^{\infty}\px\Pe \ue^2 dx +\frac{1}{2}\ges^2(t)\int_{0}^{\infty}\Pe dx - \eps\px\ue(t,0)\int_{0}^{\infty}\Pe dx \\
&\qquad +\eps\int_{0}^{\infty}\Pe\px\ue dx+\eps^2\int_{0}^{\infty}\px\Pe\px\ue dx -\eps\ptx\Pe(t,0)\int_{0}^{\infty}\Pe dx.
\end{split}
\end{equation}
Due to \eqref{eq:def-di-f} and \eqref{eq:lim-di-f},
\begin{align*}
2\gamma\int_{0}^{\infty}\Pe\Fe dx&=2\gamma\int_{0}^{\infty}\Fe\px\Fe dx=\gamma(\Fe(t,\infty))^2\\
&= \frac{1}{\gamma}\left(\eps\ptx\Pe(t,0)+\eps\px\ue(t,0)-\frac{1}{2}\ges^2(t)\right)^2,
\end{align*}
that is
\begin{equation}
\begin{split}
\label{eq:342}
2\gamma\int_{0}^{\infty}\Pe\Fe dx=& \frac{\eps^2}{\gamma}(\ptx\Pe(t,0))^2+ \frac{2\eps^2}{\gamma}\ptx\Pe(t,0)\px\ue(t,0)+\frac{\eps^2}{\gamma}(\px\ue(t,0))^2\\
&+\frac{1}{4\gamma}\ges^4(t) -\frac{\eps}{\gamma}\ptx\Pe(t,0)\ges^2(t)-\frac{\eps}{\gamma}\px\ue(t,0)\ges^2(t).
\end{split}
\end{equation}
Again by \eqref{eq:lim-di-f},
\begin{equation}
\label{eq:343}
\begin{split}
-2\eps\px\ue(t,0)\int_{0}^{\infty}\Pe dx=& -2\frac{\eps^2}{\gamma}(\ptx\Pe(t,0))\px\ue(t,0)\\ &-2\frac{\eps^2}{\gamma}(\px\ue(t,0))^2+\frac{\eps}{\gamma}\px\ue(t,0)\ges^2(t),\\
- 2\eps\ptx\Pe(t,0)\int_{0}^{\infty}\Pe dx=& -2\frac{\eps^2}{\gamma}(\ptx\Pe(t,0))^2\\
&- 2\frac{\eps^2}{\gamma}\ptx\Pe(t,0)\px\ue(t,0)+\frac{\eps}{\gamma}\ptx\Pe(t,0)\ges^2(t),\\
\ges^2(t)\int_{0}^{\infty}\Pe dx=& \frac{\eps}{\gamma}\ptx\Pe(t,0)\ges^2(t)+\frac{\eps}{\gamma}\px\ue(t,0)\ges^2(t)-\frac{1}{2\gamma}\ges^4(t).
\end{split}
\end{equation}
Therefore, \eqref{eq:12312}, \eqref{eq:342} and \eqref{eq:343} give
\begin{align*}
&\frac{d}{dt}\left(\int_{0}^{\infty}\Pe^2dx+ \eps^2\int_{0}^{\infty}(\px\Pe)^2dx\right)\\
&\quad=\frac{\eps^2}{\gamma}(\ptx\Pe(t,0))^2+ \frac{2\eps^2}{\gamma}\ptx\Pe(t,0)\px\ue(t,0)+ \frac{\eps^2}{\gamma}(\px\ue(t,0))^2\\
&\qquad +\frac{1}{4\gamma}\ges^4(t) -\frac{\eps}{\gamma}\ptx\Pe(t,0)\ges^2(t)-\frac{\eps}{\gamma}\px\ue(t,0)\ges^2(t)\\
&\qquad -2\eps\gamma\int_{0}^{\infty}\px\Pe\Fe dx - \int_{0}^{\infty} \Pe \ue^2 dx +\eps \int_{0}^{\infty}\px\Pe \ue^2 dx\\
&\qquad +\frac{\eps}{\gamma}\ptx\Pe(t,0)\ges^2(t)+\frac{\eps}{\gamma}\px\ue(t,0)\ges^2(t)-\frac{1}{2\gamma}\ges^4(t)\\
&\qquad -2\frac{\eps^2}{\gamma}(\ptx\Pe(t,0))\px\ue(t,0) -2\frac{\eps^2}{\gamma}(\px\ue(t,0))^2 +\frac{\eps}{\gamma}\px\ue(t,0)\ges^2(t) \\
&\qquad + 2\eps\int_{0}^{\infty}\Pe\px\ue dx+2\eps^2\int_{0}^{\infty}\px\Pe\px\ue dx  -2\frac{\eps^2}{\gamma}(\ptx\Pe(t,0))^2\\
&\qquad - 2\frac{\eps^2}{\gamma}\ptx\Pe(t,0)\px\ue(t,0)+\frac{\eps}{\gamma}\ptx\Pe(t,0)\ges^2(t),
\end{align*}
that is,
\begin{equation}
\label{eq:345}
\begin{split}
&\frac{d}{dt}\left(\int_{0}^{\infty}\Pe^2dx+ \eps^2\int_{0}^{\infty}(\px\Pe)^2dx\right)+\frac{1}{\gamma}\left(\eps\ptx\Pe(t,0)+\eps\px\ue(t,0)-\frac{1}{2}\ges^2(t)\right)^2\\
&\quad =  -2\eps\gamma\int_{0}^{\infty}\px\Pe\Fe dx - \int_{0}^{\infty} \Pe \ue^2dx +\eps \int_{0}^{\infty}\px\Pe \ue^2 dx\\
&\qquad  + 2\eps\int_{0}^{\infty}\Pe\px\ue dx +2\eps^2\int_{0}^{\infty}\px\Pe\px\ue dx.
\end{split}
\end{equation}
Thanks to \eqref{eq:OHepsw}, \eqref{eq:P-pxP-intfy}, \eqref{eq:def-di-f} and \eqref{eq:lim-di-f},
\begin{equation}
\label{eq:346}
-2\eps\gamma\int_{0}^{\infty}\px\Pe\Fe dx=2\eps\gamma\int_{0}^{\infty}\Pe\px\Fe dx = 2\eps\gamma\int_{0}^{\infty} \Pe^2 dx\le 2\gamma \int_{0}^{\infty} \Pe^2 dx,
\end{equation}
while, for \eqref{eq:OHepsw} and \eqref{eq:P-pxP-intfy},
\begin{equation}
\label{eq:347}
 2\eps\int_{0}^{\infty}\Pe\px\ue=-2\eps\int_{0}^{\infty}\ue\px\Pe dx.
\end{equation}
Hence, \eqref{eq:345}, \eqref{eq:346} and \eqref{eq:347} give
\begin{equation*}
\begin{split}
&\frac{d}{dt}\left(\norm{\Pe(t,\cdot)}^2_{L^{2}(0,\infty)}+\eps^2\norm{\px\Pe(t,\cdot)}^2_{L^{2}(0,\infty)}\right)\\
&\qquad +\frac{1}{\gamma}\left(\eps\ptx\Pe(t,0)+\eps\px\ue(t,0)-\frac{1}{2}\ges^2(t)\right)^2\\
&\quad \leq 2\gamma\norm{\Pe(t,\cdot)}^2_{L^{2}(0,\infty)} - \int_{0}^{\infty} \Pe \ue^2dx +\eps \int_{0}^{\infty}\px\Pe \ue^2 dx\\
&\qquad  - 2\eps\int_{0}^{\infty}\ue\px\Pe dx +2\eps^2\int_{0}^{\infty}\px\Pe\px\ue dx.
\end{split}
\end{equation*}
Thus,
\begin{equation}
\label{eq:349}
\begin{split}
&\frac{d}{dt}\left(\norm{\Pe(t,\cdot)}^2_{L^{2}(0,\infty)}+\eps^2\norm{\px\Pe(t,\cdot)}^2_{L^{2}(0,\infty)}\right)\\
&\qquad +\frac{1}{\gamma}\left(\eps\ptx\Pe(t,0)+\eps\px\ue(t,0)-\frac{1}{2}\ges^2(t)\right)^2\\
&\quad \leq 2\gamma\norm{\Pe(t,\cdot)}^2_{L^{2}(0,\infty)} +\left\vert \int_{0}^{\infty} \Pe \ue^2dx\right\vert +\eps \left\vert\int_{0}^{\infty}\px\Pe \ue^2 dx\right\vert\\
&\qquad +2\eps\left\vert \int_{0}^{\infty}\ue\px\Pe dx\right\vert +2\eps^2\left\vert\int_{0}^{\infty}\px\Pe\px\ue dx\right\vert\\
&\quad  \leq 2\gamma\norm{\Pe(t,\cdot)}^2_{L^{2}(0,\infty)}+ \int_{0}^{\infty} \vert\Pe\vert \ue^2dx +\eps \int_{0}^{\infty}\vert\px\Pe\vert \ue^2 dx\\
&\qquad +2\eps\int_{0}^{\infty}\vert\ue\vert\vert\px\Pe\vert dx + 2\eps^2\int_{0}^{\infty}\vert\px\Pe\vert\vert\px\ue\vert dx.
\end{split}
\end{equation}
For Young's inequality,
\begin{align*}
&2\eps\int_{0}^{\infty}\vert\px\Pe \vert\vert \ue\vert dx=\int_{0}^{\infty}\left\vert\frac{\ue}{\sqrt{\gamma}}\right\vert\vert 2\eps\sqrt{\gamma}\px\Pe\vert dx\\
&\quad \le  2\gamma\eps^2\norm{\px\Pe(t,\cdot)}^2_{L^{2}(0,\infty)} +\frac{1}{2\gamma}\norm{\ue(t,\cdot)}^2_{L^{2}(0,\infty)},\\
&2\eps^2\int_{0}^{\infty}\vert\px\Pe\vert\vert\px\ue\vert \le\eps^2\norm{\px\Pe(t,\cdot)}^2_{L^{2}(0,\infty)}+ \eps^2 \norm{\px\ue(t,\cdot)}^2_{L^{2}(0,\infty)}.
\end{align*}
Thus,
\begin{align*}
&\frac{d}{dt}G(t)+\frac{1}{\gamma}\left(\eps\ptx\Pe(t,0)+\eps\px\ue(t,0)-\frac{1}{2}\ges^2(t)\right)^2\\
&\quad \le 2\gamma G(t)+\frac{1}{2\gamma}\norm{\ue(t,\cdot)}^2_{L^{2}(0,\infty)} + \int_{0}^{\infty} \vert\Pe\vert \ue^2 dx + \eps  \int_{0}^{\infty}\vert\px\Pe\vert \ue^2dx\\
&\qquad +\eps^2 \int_{0}^{\infty}(\px\ue)^2 dx +\eps^2\norm{\px\Pe(t,\cdot)}^2_{L^{2}(0,\infty)}+ \eps^2 \norm{\px\ue(t,\cdot)}^2_{L^{2}(0,\infty)},
\end{align*}
that is
\begin{equation}
\label{eq:350}
\begin{split}
&\frac{d}{dt}G(t)-2\gamma G(t)+\frac{1}{\gamma}\left(\eps\ptx\Pe(t,0)+\eps\px\ue(t,0)-\frac{1}{2}\ges^2(t)\right)^2\\
&\quad \le \frac{1}{2\gamma}\norm{\ue(t,\cdot)}^2_{L^{2}(0,\infty)}+\int_{0}^{\infty} \vert\Pe\vert \ue^2 dx\\
&\qquad + \eps  \int_{0}^{\infty}\vert\px\Pe\vert \ue^2dx +\eps^2\norm{\px\Pe(t,\cdot)}^2_{L^{2}(0,\infty)}\\
&\qquad +\eps^2 \norm{\px\ue(t,\cdot)}^2_{L^{2}(0,\infty)},
\end{split}
\end{equation}
where
\begin{equation}
\label{eq:def-di-G}
G(t)=\norm{\Pe(t,\cdot)}^2_{L^2(0,\infty)} + \eps^2\norm{\px\Pe(t,\cdot)}^2_{L^2(0,\infty)}.
\end{equation}
We observe that, for \eqref{eq:l-2-u},
\begin{equation}
\label{eq:399}
\int_{0}^{\infty} \vert\Pe\vert \ue^2 dx\le C_{0}\left(e^{C_0 t}(1+t)+1\right) \norm{\Pe}_{L^{\infty}(I_T)},
\end{equation}
where $I_T$ is defined in \eqref{eq:def-di-I}.\\
Since $0<\eps<1$, it follows from \eqref{eq:l-2-u} and \eqref{eq:10021} that
\begin{equation}
\label{eq:400}
\begin{split}
\eps\int_{0}^{\infty}\vert\px\Pe\vert \ue^2dx&\le \eps\norm{\px\Pe(t,\cdot)}_{L^{\infty}(0,\infty)}\norm{\ue(t,\cdot)}^2_{L^2(0,\infty)}\\
&\le \sqrt{\eps}C_{0}\left(e^{C_{0}t}(1+t) +1\right)^{\frac{3}{2}}\le C_{0}\left(e^{C_{0}t}(1+t) +1\right)^{\frac{3}{2}}.
\end{split}
\end{equation}
Again by $0<\eps <1$ and \eqref{eq:10021}, we have that
\begin{equation}
\label{eq:401}
\eps^2\int_{0}^{\infty}(\px\Pe)^2 dx\le\norm{\px\Pe(t,\cdot)}^2_{L^2(0,\infty)} \le C_{0}\left(e^{C_0 t}(1+t)+1\right).
\end{equation}
Therefore, \eqref{eq:l-2-u}, \eqref{eq:350}, \eqref{eq:399}, \eqref{eq:400} and \eqref{eq:401} give
\begin{equation}
\label{eq:502}
\begin{split}
&\frac{d}{dt}G(t)-2\gamma G(t)+\frac{1}{\gamma}\left(\eps\ptx\Pe(t,0)+\eps\px\ue(t,0)-\frac{1}{2}\ges^2(t)\right)^2\\
&\quad \le \theta_{1}(t) + \theta_{2}(t) \norm{\Pe}_{L^{\infty}(I_T)}+ \eps^2\norm{\px\ue(t,\cdot)}^2_{L^2(0,\infty)},
\end{split}
\end{equation}
where
\begin{align*}
\theta_{1}(t)=&2C_{0}\left(e^{C_0 t}(1+t)+1\right)+C_{0}\left(e^{C_{0}t}(1+t) +1\right)^{\frac{3}{2}},\\
\theta_{2}(t)=& C_{0}\left(e^{C_0 t}(1+t)+1\right),
\end{align*}
are two continuous functions in $t$.

Gronwall's Lemma, \eqref{eq:u0eps}  and \eqref{eq:def-di-G} give
\begin{align*}
&\norm{\Pe(t,\cdot)}^2_{L^2(0,\infty)} + \eps^2\norm{\px\Pe(t,\cdot)}^2_{L^2(0,\infty)}\\
&\qquad +\frac{e^{2\gamma t}}{\gamma}\int_{0}^{t} e^{-2\gamma s} \left(\eps\ptx\Pe(s,0)+\eps\px\ue(s,0)-\frac{1}{2}\ges^2(s)\right)^2 ds\\
&\quad\le \norm{P_{0}}^2_{L^2(0,\infty)}e^{2\gamma t}+e^{2\gamma t}\int_{0}^{t}e^{-2\gamma s}\theta_{1}(s)ds + \norm{\Pe}_{L^{\infty}(I_T)}e^{2\gamma t}\int_{0}^{t}e^{-2\gamma s}\theta_{2}(s)ds\\
&\qquad+\eps^2 e^{2\gamma t}\int_{0}^{t} e^{-2\gamma t} \norm{\px\ue(s,\cdot)}^2_{L^2(0,\infty)} ds\\
&\quad\le \norm{P_{0}}^2_{L^2(0,\infty)}e^{2\gamma t} +\norm{\theta_{1}}_{L^{\infty}(0,T)}te^{2\gamma t} + \norm{\Pe}_{L^{\infty}(I_T)}\norm{\theta_{2}}_{L^{\infty}(0,T)}t e^{2\gamma t}\\
&\qquad+\eps^2 e^{2\gamma t}\int_{0}^{t} e^{-2\gamma } \norm{\px\ue(s,\cdot)}^2_{L^2(0,\infty)} ds.
\end{align*}
For \eqref{eq:ux-l-2},
\begin{align*}
&\eps^2 e^{2\gamma t}\int_{0}^{t} e^{-2\gamma s} \norm{\px\ue(s,\cdot)}^2_{L^2(0,\infty)} ds\\
&\quad\le \eps e^{2\gamma t}\int_{0}^{t}\norm{\px\ue(s,\cdot)}^2_{L^2(0,\infty)} ds\le\eps\theta_{3}(t)\le\norm{\theta_3}_{L^{\infty}(0,T)},
\end{align*}
where,
\begin{equation*}
\theta_{3}(t)=C_{0}e^{2\gamma t}\left(e^{C_0 t}(1+t)+t\right).
\end{equation*}
Hence,
\begin{align*}
&\norm{\Pe(t,\cdot)}^2_{L^2(0,\infty)} + \eps^2\norm{\px\Pe(t,\cdot)}^2_{L^2(0,\infty)}\\
&\qquad+\frac{e^{2\gamma t}}{\gamma}\int_{0}^{t} e^{-2\gamma s} \left(\eps\ptx\Pe(s,0)+\eps\px\ue(s,0)-\frac{1}{2}\ges^2(s)\right)^2 ds\\
&\quad \le \norm{P_{0}}^2_{L^2(0,\infty)}e^{2\gamma t} +\norm{\theta_{1}}_{L^{\infty}(0,T)}te^{2\gamma t} + \norm{\Pe}_{L^{\infty}(I_T)}\norm{\theta_{2}}_{L^{\infty}(0,T)}t e^{2\gamma t}+ \norm{\theta_3}_{L^{\infty}(0,T)}
\end{align*}
that is
\begin{equation}
\label{eq:l-2-P}
\begin{split}
&\norm{\Pe(t,\cdot)}^2_{L^2(0,\infty)} + \eps^2\norm{\px\Pe(t,\cdot)}^2_{L^2(0,\infty)}\\
&\qquad+\frac{e^{2\gamma t}}{\gamma}\int_{0}^{t} e^{-2\gamma s} \left(\eps\ptx\Pe(s,0)+\eps\px\ue(s,0)-\frac{1}{2}\ges^2(s)\right)^2 ds\\
&\quad \le  C(T)\left(\norm{\Pe}_{L^{\infty}(I_T)}+1\right).
\end{split}
\end{equation}
Due to \eqref{eq:OHepsw}, \eqref{eq:10021}, \eqref{eq:l-2-P} and the H\"older inequality,
\begin{align*}
\Pe^2(t,x)&\le 2\int_{0}^{\infty}\vert\Pe\vert\vert\px\Pe\vert dx \le 2\norm{\Pe(t,\cdot)}_{L^2(0,\infty)} \norm{\px\Pe(t,\cdot)}_{L^2(0,\infty)}\\
&\le 2\sqrt{C(T)\left(\norm{\Pe}_{L^{\infty}(I_T)}+1\right)}\sqrt{C_{0}\left(e^{C_0 t}(1+t)+1\right)}\\
&\le C(T)\left(\norm{\Pe}_{L^{\infty}(I_T)}+1\right).
\end{align*}
Therefore,
\begin{equation}
\label{eq:equa-P-norm}
\norm{\Pe}^2_{L^{\infty}(I_T)} - C(T) \norm{\Pe}_{L^{\infty}(I_T)} - C(T)\le 0,
\end{equation}
which gives \eqref{eq:l-infty-P-1}.

\eqref{eq:l-2-P-1}, \eqref{eq:L-2-px-P-1}, \eqref{eq:l-2-px-u} follow from \eqref{eq:l-infty-P-1} and \eqref{eq:l-2-P}.

Let us show that \eqref{eq:pt-px-P} holds true.
Multiplying \eqref{eq:equat-per-P} by $\Pe$, an integration on $(0,\infty)$ gives
\begin{align*}
2\eps\int_{0}^{\infty}\Pe\ptx\Pe dx=& \frac{d}{dt}\norm{\Pe(t,\cdot)}^2_{L^2(0,\infty)}-2\gamma\int_{0}^{\infty}\Pe\Fe dx +\int_{0}^{\infty}\Pe \ue^2 dx\\
&-\ges^2(t)\int_{0}^{\infty}\Pe dx  -2\eps\int_{0}^{\infty}\Pe\px\ue dx \\
& + 2\eps \px\ue(t,0)\int_{0}^{\infty}\Pe dx+ 2\eps\ptx\Pe(t,0)\int_{0}^{\infty}\Pe dx.
\end{align*}
It follows from \eqref{eq:def-di-f}, \eqref{eq:lim-di-f}, \eqref{eq:342} and \eqref{eq:343} that
\begin{align*}
2\eps\int_{0}^{\infty}\Pe\ptx\Pe dx=& \frac{d}{dt}\norm{\Pe(t,\cdot)}^2_{L^2(0,\infty)}- \frac{\eps^2}{\gamma}(\ptx\Pe(t,0))^2\\
&- \frac{2\eps^2}{\gamma}\ptx\Pe(t,0)\px\ue(t,0)- \frac{\eps^2}{\gamma}(\px\ue(t,0))^2\\
&-\frac{1}{4\gamma}\ges^4(t) +\frac{\eps}{\gamma}\ptx\Pe(t,0)\ges^2(t)+\frac{\eps}{\gamma}\px\ue(t,0)\ges^2(t)\\
& +\int_{0}^{\infty}\Pe \ue^2dx -2\eps\int_{0}^{\infty}\Pe\px\ue dx\\
&-\frac{\eps}{\gamma}\ptx\Pe(t,0)\ges^2(t)-\frac{\eps}{\gamma}\px\ue(t,0)\ges^2(t)+\frac{1}{2\gamma}\ges^4(t)\\
&+2\frac{\eps^2}{\gamma}\ptx\Pe(t,0)\px\ue(t,0) +2\frac{\eps^2}{\gamma}(\px\ue(t,0))^2-\frac{\eps}{\gamma}\px\ue(t,0)\ges^2(t)\\
&+2\frac{\eps^2}{\gamma}(\ptx\Pe(t,0))^2 +2\frac{\eps^2}{\gamma}\ptx\Pe(t,0)\px\ue(t,0)-\frac{\eps}{\gamma}\ptx\Pe(t,0)\ges^2(t),
\end{align*}
that is,
\begin{align*}
2\eps\int_{0}^{\infty}\Pe\ptx\Pe dx=& \frac{d}{dt}\norm{\Pe(t,\cdot)}^2_{L^2(0,\infty)}\\
&+\frac{1}{\gamma}\left(\eps\ptx\Pe(t,0)+\eps\px\ue(t,0)-\frac{1}{2}\ges(t)^2\right)^2\\
&+\int_{0}^{\infty}\Pe \ue^2dx -2\eps\int_{0}^{\infty}\Pe\px\ue dx.
\end{align*}
An integration on $(0,t)$ gives
\begin{align*}
2\eps\int_{0}^{t}\!\!\!\int_{0}^{\infty}\Pe\ptx\Pe dsdx=&\norm{\Pe(t,\cdot)}^2_{L^2(0,\infty)}-\norm{P_{\eps,0}}^2_{L^2(0,\infty)}\\
&+\frac{1}{\gamma}\int_{0}^{t}\left(\eps\ptx\Pe(s,0)+\eps\px\ue(s,0)-\frac{1}{2}\ges^2(s)\right)^2 ds\\
&+\int_{0}^{t}\!\!\!\int_{0}^{\infty}\Pe \ue^2 dx -2\eps\int_{0}^{t}\!\!\!\int_{0}^{\infty}\Pe\px\ue dsdx.
\end{align*}
It follows from \eqref{eq:u0eps}, \eqref{eq:l-2-P-1} and \eqref{eq:399} that
\begin{align*}
2\eps\left\vert\int_{0}^{t}\!\!\!\int_{0}^{\infty}\Pe\ptx\Pe dsdx\right\vert\le&\norm{\Pe(t,\cdot)}^2_{L^2(0,\infty)}+\norm{P_{\eps,0}}^2_{L^2(0,\infty)}\\       &+\frac{1}{\gamma}\int_{0}^{t}\left(\eps\ptx\Pe(s,0)+\eps\px\ue(s,0)-\ges^2(s)\right)^2 ds\\
&+2\eps\int_{0}^{t}\!\!\!\int_{0}^{\infty}\vert\Pe\vert\vert\px\ue\vert dx+C(T)\\
\le&\norm{P_{0}}^2_{L^2(0,\infty)}\\
&+\frac{e^{2\gamma t}}{\gamma}\int_{0}^{t}e^{-2\gamma s}\left(\eps\ptx\Pe(s,0)+\eps\px\ue(s,0)-\frac{1}{2}\ges^2(s)\right)^2 ds\\
&+2\eps\int_{0}^{t}\!\!\!\int_{0}^{\infty}\vert\Pe\vert\vert\px\ue\vert dsdx+C(T)\\
\le & \norm{P_{0}}^2_{L^2(0,\infty)} +2\eps\int_{0}^{t}\!\!\!\int_{0}^{\infty}\vert\Pe\vert\vert\px\ue\vert dsdx+C(T)
\end{align*}
Due to \eqref{eq:l-2-P-1} and Young's inequality,
\begin{equation}
\label{eq:young}
\begin{split}
&2\eps\int_{0}^{\infty}\vert\Pe\vert\vert\px\ue\vert dx=2\int_{0}^{\infty}\vert\Pe\vert\vert\eps\px\ue\vert\\
&\quad\le \norm{\Pe(t,\cdot)}^2_{L^2(0,\infty)}+\eps^2\norm{\px\ue(t,\cdot)}^2_{L^2(0,\infty)}\\
&\quad\le C(T) + \eps^2\norm{\px\ue(t,\cdot)}^2_{L^2(0,\infty)}
\end{split}
\end{equation}
Thus, for \eqref{eq:ux-l-2} and \eqref{eq:young}, we have that
\begin{align*}
2\eps\int_{0}^{t}\!\!\!\int_{0}^{\infty}\vert\Pe\vert\vert\px\ue\vert dsdx\le \int_{0}^{t}\norm{\Pe(s,\cdot)}^2_{L^2(0,\infty)} ds +  \eps^2\int_{0}^{t}\norm{\px\ue(s,\cdot)}^2_{L^2(0,\infty)} ds\le C(T)  .
\end{align*}
Therefore,
\begin{equation*}
2\eps\left\vert\int_{0}^{t}\!\!\!\int_{0}^{\infty}\Pe\ptx\Pe dsdx\right\vert\le \norm{P_{0}}^2_{L^2(0,\infty)}+C(T),
\end{equation*}
which gives \eqref{eq:pt-px-P}.
\end{proof}

\begin{lemma}
\label{lm:linfty-u}
Let $T>0$. Then,
\begin{equation}
\label{eq:linfty-u}
\norm{\ue}_{L^\infty(I_{T})}\le\norm{u_0}_{L^\infty(0,\infty)}+C(T),
\end{equation}
where $I_{T}$ is defined in \eqref{eq:def-di-I}.
\end{lemma}

\begin{proof}
Due to \eqref{eq:OHepsw} and \eqref{eq:l-infty-P-1},
\begin{equation*}
\pt \ue +\ue\px\ue-\eps\pxx \ue\le \gamma C(T).
\end{equation*}
Since the map
\begin{equation*}
{\mathcal F}(t):=\norm{u_0}_{L^\infty(0,\infty)}+\gamma C(T)t,
\end{equation*}
solves the equation
\begin{equation*}
\frac{d{\mathcal F}}{dt}=\gamma C(T)
\end{equation*}
and
\begin{equation*}
\max\{\ue(0,x),0\}\le {\mathcal F}(t),\qquad (t,x)\in I_T,
\end{equation*}
the comparison principle for parabolic equations implies that
\begin{equation*}
 \ue(t,x)\le {\mathcal F}(t),\qquad (t,x)\in I_T.
\end{equation*}
In a similar way we can prove that
\begin{equation*}
\ue(t,x)\ge -{\mathcal F}(t),\qquad (t,x)\in I_T.
\end{equation*}
Therefore,
\begin{equation*}
\vert\ue(t,x)\vert\le\norm{u_0}_{L^\infty(0,\infty)}+\gamma C(T)t\le  \norm{u_0}_{L^\infty(0,\infty)}+C(T),
\end{equation*}
which gives \eqref{eq:linfty-u}.
\end{proof}

\section{Proof of Theorem \ref{th:main}}
\label{sec:proof}

This section is devoted to the proof of Theorem \ref{th:main}.

Let us begin by proving the existence of  a distributional solution
to  \eqref{eq:OH}, \eqref{eq:boundary}, \eqref{eq:init}  satisfying \eqref{eq:ent2}.
\begin{lemma}\label{lm:conv}
Let $T>0$. There exists a function $u\in L^{\infty}((0,T)\times (0,\infty))$ that is a distributional
solution of \eqref{eq:OHw} and satisfies  \eqref{eq:ent2}.
\end{lemma}

We  construct a solution by passing
to the limit in a sequence $\Set{u_{\eps}}_{\eps>0}$ of viscosity
approximations \eqref{eq:OHepsw}. We use the
compensated compactness method \cite{TartarI}.

\begin{lemma}\label{lm:conv-u}
Let $T>0$. There exists a subsequence
$\{\uek\}_{k\in\N}$ of $\{\ue\}_{\eps>0}$
and a limit function $  u\in L^{\infty}((0,T)\times(0,\infty))$
such that
\begin{equation}\label{eq:convu}
    \textrm{$\uek \to u$ a.e.~and in $L^{p}_{loc}((0,T)\times(0,\infty))$, $1\le p<\infty$}.
\end{equation}
Moreover, we have
\begin{equation}
\label{eq:conv-P}
\textrm{$\Pek \to P$ a.e.~and in $L^{p}_{loc}(0,T;W^{1,p}_{loc}(0,\infty))$, $1\le p<\infty$},
\end{equation}
where
\begin{equation}
\label{eq:tildeu}
P(t,x)=\int_0^x u(t,y)dy,\qquad t\ge 0,\quad x\ge 0,
\end{equation}
and \eqref{eq:ent2} holds true.
\end{lemma}

\begin{proof}
Let $\eta:\R\to\R$ be any convex $C^2$ entropy function, and
$q:\R\to\R$ be the corresponding entropy
flux defined by $q'=f'\eta'$.
By multiplying the first equation in \eqref{eq:OHepsw} with
$\eta'(\ue)$ and using the chain rule, we get
\begin{equation*}
    \pt  \eta(\ue)+\px q(\ue)
    =\underbrace{\eps \pxx \eta(\ue)}_{=:\CLea_{1,\eps}}
    \, \underbrace{-\eps \eta''(\ue)\left(\px  \ue\right)^2}_{=: \CLea_{2,\eps}}
     \, \underbrace{+\gamma\eta'(\ue) \Pe}_{=: \CLea_{3,\eps}},
\end{equation*}
where  $\CLea_{1,\eps}$, $\CLea_{2,\eps}$, $\CLea_{3,\eps}$ are distributions.

Let us show that
\begin{equation*}
\label{eq:H1}
\textrm{$\CLea_{1,\eps}\to 0$ in $H^{-1}((0,T)\times(0,\infty))$, $T>0$.}
\end{equation*}
Since
\begin{equation*}
\eps\pxx\eta(\ue)=\px(\eps\eta'(\ue)\px\ue),
\end{equation*}
for \eqref{eq:ux-l-2} and Lemma \ref{lm:linfty-u},
\begin{align*}
\norm{\eps\eta'(\ue)\px\ue}^2_{L^2((0,T)\times (0,\infty))}&\le\eps ^2\norm{\eta'}^2_{L^{\infty}(J_T)}\int_{0}^{T}\norm{\px\ue(s,\cdot)}^2_{L^2(0,\infty)}ds\\
&\le\eps\norm{\eta'}^2_{L^{\infty}(J_T)}C(T)\to 0,
\end{align*}
where
\begin{equation*}
J_T=\left(-\norm{u_0}_{L^\infty(0,\infty)}- C(T), \norm{u_0}_{L^\infty(0,\infty)}+C(T)\right).
\end{equation*}

We claim that
\begin{equation*}
\label{eq:L1}
\textrm{$\{\CLea_{2,\eps}\}_{\eps>0}$ is uniformly bounded in $L^1((0,T)\times(0,\infty))$, $T>0$}.
\end{equation*}
Again by \eqref{eq:ux-l-2} and Lemma \ref{lm:linfty-u},
\begin{align*}
\norm{\eps\eta''(\ue)(\px\ue)^2}_{L^1((0,T)\times (0,\infty))}&\le
\norm{\eta''}_{L^{\infty}(J_T)}\eps
\int_{0}^{T}\norm{\px\ue(s,\cdot)}^2_{L^2(0,\infty)}ds\\
&\le \norm{\eta''}_{L^{\infty}(J_T)}C(T).
\end{align*}
We have that
\begin{equation*}
\textrm{$\{\CL_{3,\eps}\}_{\eps>0}$ is uniformly bounded in $L^1_{loc}((0,T)\times (0,\infty))$, $T>0$.}
\end{equation*}
Let $K$ be a compact subset of $(0,T)\times (0,\infty)$. For Lemmas \ref{lm:P-infty} and \ref{lm:linfty-u},
\begin{align*}
\norm{\gamma\eta'(\ue)\Pe}_{L^1(K)}&=\gamma\int_{K}\vert\eta'(\ue)\vert\vert\Pe\vert
dtdx\\
&\leq \gamma
\norm{\eta'}_{L^{\infty}(J_T)}\norm{\Pe}_{L^{\infty}(I_{T})}\vert K \vert .
\end{align*}
Therefore, Murat's lemma \cite{Murat:Hneg} implies that
\begin{equation}
\label{eq:GMC1}
    \text{$\left\{  \pt  \eta(\ue)+\px q(\ue)\right\}_{\eps>0}$
    lies in a compact subset of $\Hneg((0,T)\times(0,\infty))$.}
\end{equation}
The $L^{\infty}$ bound stated in Lemma \ref{lm:linfty-u}, \eqref{eq:GMC1}, and the
 Tartar's compensated compactness method \cite{TartarI} give the existence of a subsequence
$\{\uek\}_{k\in\N}$ and a limit function $  u\in L^{\infty}((0,T)\times(0,\infty)),\,T>0,$
such that \eqref{eq:convu} holds.

Let us prove that \eqref{eq:conv-P} holds true.\\
We show that
\begin{equation}
\label{eq:px-1}
\textrm{$\eps\px\Pe(t,x)\to 0$ in $L^{\infty}(0,T; L^{\infty}(0,\infty))$, $T>0$.}
\end{equation}
It follows from \eqref{eq:10021} that
\begin{equation*}
\eps\norm{\px\Pe}_{L^{\infty}(0,T; L^{\infty}(0,\infty))}\leq \sqrt{\eps} \sqrt{C_{0}\left(e^{C_{0}T}(1+T) + 1\right)}=\sqrt{\eps}C(T)\to 0,
\end{equation*}
that is \eqref{eq:px-1}.\\
Then, \eqref{eq:P-int-in-0}, \eqref{eq:convu}, \eqref{eq:px-1} and  the H\"older inequality give \eqref{eq:conv-P}.

Finally, we prove \eqref{eq:ent2}.\\
Let $k\in\N$, $c\in\R$ be a constant, and $\phi\in C^{\infty}(\R^2)$ be a nonnegative test function with compact support. Multiplying the first equation of \eqref{eq:OHepsw} by $\sgn{\ue-c}$, we have
\begin{align*}
\pt \vert \uek -c\vert&+\px\left(\sgn{\uek-c}\left(\frac{\uek^2}{2}-\frac{c^2}{2}\right)\right)\\
&-\gamma\sgn{\uek-c}\Pek-\epsk\pxx\vert \uek -c\vert \le 0.
\end{align*}
Multiplying by $\phi$ and integrating over $(0,\infty)^2$, we get
\begin{align*}
&\int_{0}^{\infty}\!\!\!\!\int_{0}^{\infty}\left(\vert \uek -c\vert \pt\phi+\left(\sgn{\uek-c}\left(\frac{\uek^2}{2}-\frac{c^2}{2}\right)\right)\px\phi\right)dtdx\\
&\qquad +\gamma \int_{0}^{\infty}\!\!\!\!\int_{0}^{\infty}\sgn{\uek-c}\Pek dtdx -\epsk  \int_{0}^{\infty}\!\!\!\!\int_{0}^{\infty}\px\vert \uek -c\vert\px\phi dtdx\\
&\qquad + \int_{0}^{\infty}\vert u_{0}(x) -c\vert\phi(0,x) dx + \int_{0}^{\infty}\sgn{g_{\eps_{k}}(t)-c}\left(\frac{g^2_{\eps_{k}}(t)}{2}-\frac{c^2}{2}\right)\phi(t,0)dt\\
&\qquad -\epsk\int_{0}^{\infty}\px\vert \uek(t,0) -c\vert\phi(t,0)dt \ge 0.
\end{align*}
Since
\begin{equation*}
\textrm{$g_{\epsk}(t)\to g(t)$ in  $W^{1,\infty}(0,\infty)$},
\end{equation*}
thanks to Lemmas \ref{lm:stima-l-2}, \ref{lm:P-infty} and \ref{lm:linfty-u}, when $k\to \infty$, we have
\begin{align*}
&\int_{0}^{\infty}\!\!\!\!\int_{0}^{\infty}\left(\vert u -c\vert \pt\phi+\left(\sgn{u-c}\left(\frac{u^2}{2}-\frac{c^2}{2}\right)\right)\px\phi\right)dtdx\\
&\qquad +\gamma \int_{0}^{\infty}\!\!\!\!\int_{0}^{\infty}\sgn{u-c}P dtdx + \int_{0}^{\infty}\vert u_{0}(x) -c\vert\phi(0,x) dx\\
&\qquad +\int_{0}^{\infty}\sgn{g(t)-c}\left(\frac{g^2(t)}{2}-\frac{c^2}{2}\right)\phi(t,0)dt\\
&\qquad -\lim_{\epsk}\epsk\int_{0}^{\infty}\px\vert \uek(t,0) -c\vert\phi(t,0)dt \ge 0.
\end{align*}
We have to prove that (see \cite{BRN})
\begin{equation}
\label{eq:ux-in-0}
\begin{split}
&\lim_{\epsk}\epsk\int_{0}^{\infty}\px\vert \uek(t,0) -c\vert\phi(t,0)dt\\
&\qquad = \int_{0}^{\infty}\sgn{g(t)-c}\left(\frac{g^2(t)}{2}- \frac{(u^\tau_0(t))^2}{2}\right)\phi(t,0)dt.
\end{split}
\end{equation}
Let $\{\rho_{\nu}\}_{\nu\in\N}\subset C^{\infty}(\R)$ be such that
\begin{equation}
0\le \rho_{\nu} \le 1, \quad \rho_{\nu}(0)=1, \quad \vert \rho'_{\nu}\vert \le 1, \quad x\ge \frac{1}{\nu} \quad \Longrightarrow \quad \rho_{\nu}(x)=0.
\end{equation}
Using $(t,x) \mapsto \rho_{\nu}(x)\phi(t,x)$ as test function for the first equation of \eqref{eq:OHepsw} we get
\begin{align*}
&\int_{0}^{\infty}\!\!\!\!\int_{0}^{\infty}\left(\uek\pt\phi\rho_{\nu}+\frac{\uek^2}{2}\px\phi\rho_{\nu} + \frac{\uek^2}{2}\phi\rho'_{\nu}\right)dtdx +\gamma\int_{0}^{\infty}\!\!\!\!\int_{0}^{\infty}\Pek\phi\rho_{\nu} dtdx\\
&\qquad -\epsk\int_{0}^{\infty}\!\!\!\!\int_{0}^{\infty}\px\uek\left(\px\phi\rho_{\nu}+\phi\rho'_{\nu}\right)dtdx + \int_{0}^{\infty} u_{0}(x)\phi(0,x)\rho_{\nu}(x)dx\\
&\qquad +\int_{0}^{\infty} \frac{g^2_{\epsk}(t)}{2}\phi(t,0) dt -\epsk\int_{0}^{\infty}\px\uek(t,0)\phi(t,0) dt=0.
\end{align*}
As $k\to\infty$, we obtain that
\begin{align*}
&\int_{0}^{\infty}\!\!\!\!\int_{0}^{\infty}\left(u\pt\phi\rho_{\nu}+\frac{u^2}{2}\px\phi\rho_{\nu} + \frac{u^2}{2}\phi\rho'_{\nu}\right)dtdx +\gamma\int_{0}^{\infty}\!\!\!\!\int_{0}^{\infty}P\phi\rho_{\nu}dtdx\\
&\qquad + \int_{0}^{\infty} u_{0}(x)\phi(0,x)\rho_{\nu}dx +\int_{0}^{\infty} \frac{g^2(t)}{2}\phi(t,0) dt\\
&\quad  =\lim_{\epsk}\epsk\int_{0}^{\infty}\px\uek(t,0)\phi(t,0) dt.
\end{align*}
Sending $\nu\to\infty$, we get
\begin{equation*}
\lim_{\epsk}\epsk\int_{0}^{\infty}\px\uek(t,0)\phi(t,0) dt= \int_{0}^{\infty}\left(\frac{g^2(t)}{2}- \frac{(u^\tau_0(t))^2}{2}\right)\phi(t,0) dt.
\end{equation*}
Therefore, due to the strong convergence of $g_{\epsk}$ and the continuity of $g$ we have
\begin{align*}
&\lim_{\epsk}\epsk\int_{0}^{\infty}\px\vert \uek(t,0)-c\vert\phi(t,0)dt\\
&\qquad = \lim_{\epsk} \int_{0}^{\infty}\px\uek(t,0)\sgn{\uek(t,0)-c}\phi(t,0)dt\\
&\qquad = \lim_{\epsk} \int_{0}^{\infty}\px\uek(t,0)\sgn{g_{\epsk}(t)-c}\phi(t,0)dt\\
&\qquad = \int_{0}^{\infty}\sgn{g(t)-c}\left(\frac{g^2(t)}{2}- \frac{(u^\tau_0(t))^2}{2}\right)\phi(t,0) dt,
\end{align*}
that is \eqref{eq:ux-in-0}.
\end{proof}

\begin{proof}[Proof of Theorem \ref{th:main}]
Lemma \eqref{lm:conv-u} gives the existence of entropy  solution $u(t,x)$ of \eqref{eq:OHw-u}, or
equivalently \eqref{eq:OHw}.

Let us show that $u(t,x)$ is unique, and that \eqref{eq:stability} holds true. Fixed $T>0$, since our solutions are bounded in $L^{\infty}((0,T)\times\R)$, we use the doubling of variables method.

Let $u,v\in L^{\infty}((0,T)\times\R)$ be two entropy solutions of \eqref{eq:OHw-u}, or equivalently of \eqref{eq:OHw}. By arguing as in \cite{BRN, CdK, dR, k}, using the fact that the two
solutions satisfy the same boundary conditions, we  prove that
\begin{equation}
\label{eq:DB}
\pt(\vert u - v \vert ) + \px\left(\frac {u^2}{2} - \frac{v^2}{2}\right)\sgn{u-v}) - \gamma\sgn{u-v} (P_{u}- P_{v})\leq 0
\end{equation}
holds in sense of distributions in $(0,\infty)\times(0,\infty)$, where
\begin{equation}
\label{eq:defdiP}
P_{u}(t,x)=\int_{0}^x u(t,y) dy, \quad P_{v}=\int_{0}^{x} v(t,y) dy.
\end{equation}
Let $\phi(t,\tau,x,y)\in C^{\infty}(\R^4) $ be a non-negative
test function such that $supp(\phi)\subset (0,\infty)^4$.
Since $u,v$ are entropy solutions of \eqref{eq:OHw-u}, we have
\begin{equation}
\label{eq:entropic12}
\begin{split}
\int_{0}^{\infty}\!\!\!\!\int_{0}^{\infty}[\vert u(t,x) &- v(\tau,y)\vert\pt\phi(t,\tau,x,y)\\
&+\left(\frac{u^2(t,x)}{2}-\frac{v^2(\tau,y)}{2}\right)\sgn{u(t,x)-v(\tau,y)}\cdot\\
&\cdot\px\phi(t,\tau,x,y)\\
&+\gamma\sgn{u(t,x)-v(\tau,y)}P_{u}(t,x)\phi(t,\tau,x,y)]dtdx\geq 0,\\
\end{split}
\end{equation}
\begin{equation}
\label{eq:entropic13}
\begin{split}
\int_{0}^{\infty}\!\!\!\!\int_{0}^{\infty}[\vert v(\tau,y) &- u(t,x)\vert\pT\phi(t,\tau,x,y)\\
&+\left(\frac{v^2(\tau,y)}{2}-\frac{u^2(x,t)}{2}\right)\sgn{v(\tau,y)-u(t,x)}\cdot\\
&\cdot\py\phi(t,\tau,x,y)\\
&+ \gamma\sgn{v(\tau,y)-u(t,x)}P_{v}(\tau,y)\phi(t,\tau,x,y)]\dtau dy\geq 0.
\end{split}
\end{equation}
Integrating \eqref{eq:entropic12} with respect to $\tau,y$,
\eqref{eq:entropic13} with respect to $t,x$, and adding these two
results, we obtain
\begin{equation*}
\begin{split}
\int_{0}^{\infty}\!\!\!\!\int_{0}^{\infty}\!\!\!\!\int_{0}^{\infty}\!\!\!\!\int_{0}^{\infty}[\vert u(t,x)&-v(\tau,y)\vert(\pt\phi(t,\tau,x,y)+\pT\phi(t,\tau,x,y))\\
&+\left(\frac{u^2(t,x)}{2}-\frac{v^2(\tau,y)}{2}\right)\sgn{u(x,t)-v(\tau,y)}\cdot\\
&\cdot(\px\phi(t,\tau,x,y)+\py\phi(t,\tau,x,y))\\
&+\gamma\sgn{u(t,x)-v(\tau,y)}(P_{u}(t,x)-P_{v}(\tau,y))\\
&\cdot\phi(t,\tau,x,y)] dt\dtau dx dy \geq 0.
\end{split}
\end{equation*}
Now, we choose a sequence of functions $\{\delta_{h}\}_{h\ge 1}$,
approximating the Dirac mass at the origin. More precisely, let
$\delta : \R\to [0,1]$ be a $ C^{\infty}$ function such that
\begin{equation*}
\int_{\R} \delta(z)dz=1,\quad  \delta (z)=0, \quad \textrm{for all}
\quad z\notin [-1,1],
\end{equation*}
and define
\begin{equation}
\label{eq:alpha1}
\delta_{h}(z)=h\delta (hz), \quad \alpha_{h}(z)=\int_{-\infty}^{z} \delta_{h}(\theta)\dtheta.
\end{equation}
Let us consider the following test function
\begin{equation}
\label{eq:phi2}
\phi_{h}(t,\tau,x,y)=\psi\Big(\frac{t+\tau}{2},\frac{x+y}{2}\Big)\delta_{h}\Big(\frac{\tau - t}{2}\Big)\delta_{h}\Big(\frac{y - x}{2}\Big),
\end{equation}
where $\psi\in C^{\infty}(\R^2) $ is a
non-negative test function such that $supp(\psi)\subset(0,\infty)^2$.\\
Using \eqref{eq:phi2} as test function in the previous inequality, we have
\begin{align*}
\int_{0}^{\infty}&\!\!\!\!\int_{0}^{\infty}\!\!\!\!\int_{0}^{\infty}\!\!\!\!\int_{0}^{\infty}\Big\{\delta_{h}\Big(\frac{\tau - t}{2}\Big)\delta_{h}\Big(\frac{y - x}{2}\Big)\Big[\vert u(t,x) -
v(\tau,y)\vert\pt\psi\Big(\frac{t+\tau}{2},\frac{x+y}{2}\Big)\\
&+ \left(\frac{u^2(t,x)}{2}-\frac{v^2(\tau,y)}{2}\right)\sgn{u(t,x)-v(\tau,y)}\px\psi\Big(\frac{t+\tau}{2},\frac{x+y}{2}\Big)\Big]\\
&+\gamma\psi\Big(\frac{t+\tau}{2},\frac{x+y}{2}\Big)\delta_{h}\Big(\frac{\tau- t}{2}\Big)\delta_{h}\Big(\frac{y - x}{2}\Big)\cdot\\
&\cdot\sgn{u(t,x)-v(\tau,y)}(P_{u}(t,x)-P_{v}(\tau,y))\Big\}dt\dtau dx dy \geq 0.
\end{align*}
We observe that $\delta_{h} \to \delta_{0}$ when $h\to 0$, where $\delta_{0}$
is Dirac mass centered in $\{0\}$. Therefore, since the maps $t \to u(t,\cdot)$, $t \to
v(t,\cdot)$ are continuous from $[0,\infty)$ into $L^{1}_{loc}(0,\infty)$,
and $t \to P_{u}(t,\cdot)$, $t \to P_{v}(t,\cdot)$ are continuous from
$[0,\infty)$ into $L^{\infty}_{loc}(0,\infty)$, it follows from  the previous inequality
that
\begin{equation}
\label{eq:1000}
\begin{split}
\int_{0}^{\infty}\!\!\!\!\int_{0}^{\infty}(\vert u -v\vert\pt\psi &+ \left(\frac{u^2}{2}-\frac{v^2}{2}\right)\sgn{u-v})\px\psi dt dx\\
& + \gamma\int_{0}^{\infty}\!\!\!\!\int_{0}^{\infty}\sgn{u-v}(P_{u}-P_{v})\psi dtdx \ge
0,
\end{split}
\end{equation}
that is \eqref{eq:DB}.

Let us show that \eqref{eq:stability} holds true.
Since $u$ is an entropy solution of \eqref{eq:OHw-u}, then it satisfies the inequality
\eqref{eq:ent2}.
We write the boundary condition in this way (see \cite{BRN}):
\begin{equation}
\label{eq:bard1}
\min_{c\in I(u^\tau_0(t),g(t))}\Big\{\sgn{u^\tau_0(t)-g(t)}\Big(\frac{(u^\tau_0(t))^2}{2}-\frac{c^2}{2}\Big)\Big\}=0,
\end{equation}
where $I(u^\tau_0(t),g(t))$ is the closed interval
$[\min\{u^\tau_0(t),g(t)\},\max\{u^\tau_0(t),g(t)\}]$.

Let us consider, now, the following product:
\begin{equation}
\label{eq:bard12}
\Big(\frac{(u^\tau_0(t))^2}{2}-\frac{c^2}{2}\Big)(\sgn{u^\tau_0(t)-c}+\sgn{c}), \quad c\in\R.
\end{equation}
We observe that \eqref{eq:bard12} is positive if $c\notin I(u^\tau_0(t),g(t))$. Instead, if we consider  $c\in I(u^\tau_0(t),g(t))$, \eqref{eq:bard12} coincides with \eqref{eq:bard1}. Therefore, for each $c\in\R$, we have that
\begin{equation}
\label{eq:bard2}
\Big(\frac{(u^\tau_0(t))^2}{2}-\frac{c^2}{2}\Big)(\sgn{u^\tau_0(t)-c}+\sgn{c})\ge 0.
\end{equation}
Since \eqref{eq:DB} holds in the sense of distributions in $(0,\infty)^2$, we have that
\begin{equation}
\label{eq:bard3}
\begin{split}
\int_{0}^{\infty}\!\!\!\!\int_{0}^{\infty}(\vert u -v\vert\pt\psi &+ \left(\frac{u^2}{2}-\frac{v^2}{2}\right)\sgn{u-v})\px\psi)dt dx\\
& + \gamma\int_{0}^{\infty}\!\!\!\!\int_{0}^{\infty}\sgn{u-v}(P_{u}-P_{v})\psi dtdx\\
\ge & \int_{0}^{\infty}\sgn{u^\tau_0(t)-v^\tau_0(t)}\cdot\\
&\cdot\left(\frac{(u^\tau_0(t))^2}{2}-\frac{(v^\tau_0(t))^2}{2}\right)\psi(t,0)dt,
\end{split}
\end{equation}
where $\psi\in C^{\infty}(\R^2)$ is a non-negative test
function with compact support, and $v^\tau_0(t)$ is the trace of $v$ at $x=0$.\\
To determine the sign of the right-hand side of \eqref{eq:bard3}, for each $t>0$, we
define the real number $c(t)$ in the following way:
\begin{equation}
\label{eq:bard4}
c(t) = \left\{ \begin{array}{ll}
u^\tau_0(t) &\textrm{if}\quad   u^\tau_0(t)\in I(g(t),v^\tau_0(t)),\\
g(t) &\textrm{if}\quad  g(t) \in I(v^\tau_0(t),u^\tau_0(t)),\\
v^\tau_0(t) &\textrm{if}\quad  v^\tau_0(t)\in I(u^\tau_0(t),g(t)).
\end{array} \right.
\end{equation}
From \eqref{eq:bard4}, it follows that
\begin{align*}
&\sgn{u^\tau_0(t)-v^\tau_0(t)}\left(\frac{(u^\tau_0(t))^2}{2}-\frac{(v^\tau_0(t))^2}{2}\right)\\
&\quad \quad \quad =\sgn{u^\tau_0(t)-c(t)}\left(\frac{(u^\tau_0(t))^2}{2}-\frac{c^2(t)}{2}\right)\\
&\quad \quad \quad\quad +\sgn{v^\tau_0(t)-c(t)}\left(\frac{(v^\tau_0(t))^2}{2}-\frac{c^2(t)}{2}\right).
\end{align*}
For \eqref{eq:bard2}, we get that the right-hand side of
\eqref{eq:bard3} is non negative. Therefore, we have \eqref{eq:1000}.

Let $T, R >0$, and let us consider the sets
\begin{equation}
\label{eq:deftrap}
\begin{split}
\Omega&:=\{(t,x)\in [0,T]\times [-R,R];\quad 0\leq s\leq t,\quad \vert x \vert\leq R + C(T)(t-s)\},\\
\Omega^+&:=\Omega\cap (0,\infty)^2,
\end{split}
\end{equation}
where
\begin{equation}
C(T)=\sup_{(0,T)\times\R}\Big\{\vert u\vert + \vert v\vert\Big\}.
\end{equation}
We define the following test function
\begin{equation*}
\label{eq:trap}
\phi_{h}(t,x)=[\alpha_{h}(s)-\alpha_{h}(s-t)][1-\alpha_{h}(\vert x \vert
-R+C(T)(t-s))]\ge 0,
\end{equation*}
where $\alpha_{h}(z)$ is defined in \eqref{eq:alpha1}.

We observe that the function $[\alpha_{h}(s)-\alpha_{h}(s-t)][1-\alpha_{h}(\vert x \vert
-R+C(T)(t-s))]$ is an approximation of the
characteristic function of $\Omega$.
Moreover, since $u$ and $v$ are in $L^{\infty}((0,T)\times\R)$, we have
that
\begin{equation}
\label{eq:f1}
\Big\vert \frac{u^2(t,x)}{2}- \frac{v^2(t,x)}{2}\Big\vert \leq C(T) \vert u(t,x) - v(t,x)\vert,
\quad (t,x)\in \Omega^+.
\end{equation}
From \eqref{eq:alpha1}, $\alpha' _{h}= \delta_{h}\ge 0$. Using
$\phi_{h}$ as test function in \eqref{eq:1000}, we have
\begin{equation*}
\begin{split}
\int_{0}^{\infty}\!\!\!\!\int_{0}^{\infty}\{\vert u &- v\vert(\delta_{h}(s)-\delta_{h}(s-t))[1-\alpha_{h}(\vert x \vert - R+C(T)(t-s))] \\
&+(\alpha_{h}(s) - \alpha_{h}(s-t))\delta_{h}(\vert x \vert-R+C(T)(t-s))\cdot\\
&\cdot[\sgn{u-v}\left(\frac{u^2}{2}-\frac{v^2}{2}\right)\sgn{x})-C(T)\vert u - v\vert]\\
&-\gamma\sgn{u-v}(P_{u}-P_{v})\cdot\\
&\cdot[\alpha_{h}(s)-\alpha_{h}(s-t)][1-\alpha_{h}(\vert x\vert-R+C(T)(t-s))]\}dsdx\ge 0.
\end{split}
\end{equation*}
Therefore, it follows from \eqref{eq:f1} and the previous inequality that
\begin{align*}
\int_{0}^{\infty}\!\!\!\!\int_{0}^{\infty}[\vert u &- v\vert (\delta_{h}(s)-\delta_{h}(s-t))[1-\alpha_{h}(\vert x \vert - R+C(T)(t-s))]\\
&-\gamma\sgn{u-v}(P_{u}-P_{v})\\
&\cdot[\alpha_{h}(s)-\alpha_{h}(s-t)][1-\alpha_{h}(\vert x\vert-R+C(T)(t-s))]]dsdx\\
\ge \int_{0}^{\infty}\!\!\!\!\int_{0}^{\infty}&(\alpha_{h}(s) - \alpha_{h}(s-t))\delta_{h}(\vert x \vert-R+C(T)(t-s))\cdot\\
& \cdot(C(T)\vert u - v\vert-\sgn{u-v}\left(\frac{u^2}{2}-\frac{v^2}{2}\right)\sgn{x}))dsdx\ge 0.
\end{align*}
Since
\begin{align*}
\delta_{h}&\to \delta_{0},\\
\alpha_{h}(\vert \cdot\vert -R+C(T)t)&\to\chi_{[-R-C(T)(t-s), R+C(T)(t-s)]},\\
[\alpha_{h}(s)-\alpha_{h}(s-t)][1-\alpha_{h}(\vert x\vert-R+C(T)(t-s))]&\to\chi_{\Omega},
\end{align*}
when $h\to 0$, where $\delta_{0}$ is Dirac mass, the continuity of $u(t,\cdot), v(t,\cdot)$ from $[0,\infty)$ into
$L^{1}_{loc}(0,\infty)$,  the continuity of $ P_{u}(t,\cdot), P_{v}(t,\cdot)$
from $[0,\infty)$ into $L^{\infty}_{loc}(0,\infty)$, and the previous inequality  give
\begin{equation}
\label{eq:1001}
\begin{split}
\norm{u(t,\cdot)-v(t,\cdot)}_{L^{1}(0,R)}\leq&\norm{u_{0}-v_{0}}_{L^{1}(0,R+C(T)t)}\\
&+\gamma\int_{\Omega^+}\sgn{u-v}(P_{u}-P_{v})dsdx\\
\leq &\norm{u_{0}-v_{0}}_{L^{1}(0,R+C(T)t)}\\
&+\gamma\int_{0}^{t}\!\!\int_{I(s)}\sgn{u-v}(P_{u}-P_{v})dsdx,
\end{split}
\end{equation}
where
\begin{equation}
\label{eq:Is}
I(s)=[0,R+C(T)(t-s)].
\end{equation}
In particular, we have
\begin{equation}
\label{eq:It}
I(t)=[0,R],\quad I(0)=[0,R+C(T)t].
\end{equation}
Therefore, it follows from \eqref{eq:1001} that
\begin{equation}
\label{eq:-1002}
\begin{split}
\norm{u(t,\cdot)-v(t,\cdot)}_{I(t)}&\leq\norm{u_{0}-v_{0}}_{I(0)}\\
+&\gamma\int_{0}^{t}\!\!\int_{I(s)}\sgn{u-v}(P_{u}-P_{v})dsdx.
\end{split}
\end{equation}
We observe that, for \eqref{eq:defdiP},
\begin{equation}
\label{eq:-1003}
\begin{split}
&\gamma\int_{0}^{t}\!\!\!\int_{I(s)}\sgn{u-v}(P_{u}-P_{v})dsdx\\
&\qquad  \le \gamma\int_{0}^{t}\!\!\!\int_{I(s)}\vert P_{u}-P_{v}\vert dsdx\\
&\qquad  \le \gamma\int_{0}^{t}\!\!\!\int_{I(s)}\Big(\Big\vert\int_{0}^{x}\vert u-v\vert dy\Big\vert\Big)dsdx\\
&\qquad  \le  \gamma\int_{0}^{t}\!\!\!\int_{I(s)}\Big(\Big\vert\int_{I(s)}\vert u-v\vert dy\Big\vert\Big)dsdx\\
&\qquad  = \gamma\int_{0}^{t} \vert I(s)\vert\norm{u(s,\cdot)-v(s,\cdot)}_{L^{1}(I(s))}ds.
\end{split}
\end{equation}
Thanks to \eqref{eq:Is}, we have
\begin{equation}
\label{eq:1005}
\vert I(s) \vert=R+C(T)(t-s)\le R+C(T)t\le R+C(T).
\end{equation}
Let us consider the following continuous function:
\begin{equation}
\label{eq:defG}
G(t)=\norm{u(t,\cdot)-v(t,\cdot)}_{L^{1}(I(t))}, \quad t\ge 0.
\end{equation}
Therefore, it follows from \eqref{eq:-1002}, \eqref{eq:-1003}, and
\eqref{eq:1005} that
\begin{equation*}
G(t)\le G(0)+C(T)\int_{0}^{t}G(s)ds.
\end{equation*}
Gronwall's Lemma, \eqref{eq:It}, and  \eqref{eq:defG} give
\begin{equation*}
\norm{u(t,\cdot)-v(t,\cdot)}_{L^{1}(0,R)}\le e^{C(T)t}\norm{u_{0}-v_{0}}_{L^{1}(0,R+C(T)t)},
\end{equation*}
that is  \eqref{eq:stability}.
\end{proof}

\end{document}